\documentclass{article}

\usepackage[latin1]{inputenc}
\usepackage{amsmath}
\usepackage{amsfonts}
\usepackage{amssymb}
\usepackage{amsthm}
\usepackage{stmaryrd} % for \ic
\usepackage{xcolor}
\usepackage{graphicx}
\usepackage{subcaption}
\usepackage{algorithmic}
\usepackage{algorithm}[0.1]

\usepackage{hyperref} % for url
\usepackage{cleveref}  % clashes in ANOR
\usepackage[colors]{optsys}
\usepackage{pifont}
\usepackage{comment}

\ifdefined\proof\else \newenvironment{proof}{\small{\bf Proof.}}{\hfill$\Box$\normalsize\bigskip} \fi

\newcounter{myenumerate}
\renewcommand{\themyenumerate}{\alph{myenumerate}}

\newcommand{\Pro}{\mathbb{P}}

\ifdefined\epi\else
\newcommand{\epi}{\mathop{\mathrm{epi}}}
\newcommand{\dom}{\mathop{\mathrm{dom}}}

\newcommand{\argmin}{\mathop{\arg\min}}
\newcommand{\argmax}{\mathop{\arg\max}}
\newcommand{\eqsepv}{\; , \enspace}
\newcommand{\eqfinv}{\; ,}
\newcommand{\eqfinp}{\; .}

\newcommand{\norm}[1]{\left\lVert#1\right\rVert}

\newcommand{\np}[1]{(#1)}
\newcommand{\bp}[1]{\big(#1\big)}
\newcommand{\Bp}[1]{\Big(#1\Big)}
\newcommand{\bgp}[1]{\bigg(#1\bigg)}
\newcommand{\Bgp}[1]{\Bigg(#1\Bigg)}

\newcommand{\nc}[1]{[#1]}
\newcommand{\bc}[1]{\big[#1\big]}
\newcommand{\Bc}[1]{\Big[#1\Big]}

\newcommand{\na}[1]{\left\{#1\right\}}
\newcommand{\ba}[1]{\big\{#1\big\}}
\newcommand{\Ba}[1]{\Big\{#1\Big\}}

\newcommand{\nset}[2]{\na{#1\,|\, #2}}
\newcommand{\bset}[2]{\ba{#1\,\big|\, #2}}
\newcommand{\Bset}[2]{\Ba{#1\,\Big|\, #2}}

\fi

\newcommand{\ce}[1]{[\![#1]\!]}         % Intervalle d'entiers

\newcommand{\B}{\mathcal{B}} % Opérateurs de Bellman

\newcommand{\func}{\phi} % Fonctionnelle
\newcommand{\Func}{F}
\newcommand{\Funcb}{\mathbf{F}} % Ensemble specifiques de fonctionnelles
\newcommand{\R}{\mathbb{R}} % Ensemble des reels
\newcommand{\N}{\mathbb{N}} % Ensemble des entiers naturels
\newcommand{\U}{\mathbb{U}} % Ensemble des controles
\newcommand{\X}{\mathbb{X}} % Ensemble des etats
\newcommand{\vopt}{\mathcal{V}}

\ifdefined\theorem\else \newtheorem{theorem}{Theorem}\fi
\ifdefined\proposition\else \newtheorem{proposition}[theorem]{Proposition}\fi
\ifdefined\definition\else \newtheorem{definition}[theorem]{Definition}\fi
\ifdefined\lemma\else \newtheorem{lemma}[theorem]{Lemma}\fi
\ifdefined\corollary\else \newtheorem{corollary}[theorem]{Corollary}\fi
\ifdefined\remark\else \newtheorem{remark}[theorem]{Remark}\fi
\ifdefined\assumption\else \newtheorem{assumption}{Assumption}\fi
\ifdefined\example\else \newtheorem{example}{Example}\fi
\ifdefined\notations\else \newtheorem{notations}{Notations}\fi

\newcommand{\cmark}{\ding{51}}%
\newcommand{\xmark}{\ding{55}}%

\newcommand{\lvopt}{\underline{\mathcal{V}}}
\newcommand{\uvopt}{\overline{\mathcal{V}}}
\newcommand{\STATESPACE}{\mathbb{X}}
\newcommand{\projname}{\pi} % euclidean projector

\newcommand{\E}{\mathbb{E}}
\newcommand{\mgraph}{\mathrm{Graph}\,}
\newcommand{\ufunc}{\overline{\phi}} %u(pper) function
\newcommand{\lfunc}{\underline{\phi}} %l(ower) function
\newcommand{\lFunc}{\underline{F}}
\newcommand{\uFunc}{\overline{F}}
\newcommand{\lFuncb}{\underline{\mathbf{F}}}
\newcommand{\uFuncb}{\overline{\mathbf{F}}}
\newcommand{\Rb}{\overline{\R}}

\newcommand{\ZSelection}[2][]{ % Selection function)
  \def\tst{#1}
  \def\testt{#2}
  \ifx\testt\empty
    \ifx\tst\empty
      S_t
    \else
      S_t^{#1}
    \fi
  \else
    S_t^{#1}\left( #2 \right)
  \fi
}
\newcommand{\lSelection}[2][t]{ % l(ower)Selection function)
  \def\tst{#1}
  \def\testt{#2}
  \ifx\tst\empty
    \ifx\testt\empty
      \underline{S}_t
    \else
      \underline{S}_t\left( #2 \right)
    \fi
  \else
    \underline{S}_{#1}\left( #2 \right)
  \fi
}
\newcommand{\uSelection}[2][t]{ % u(pper)Selection function)
  \def\tst{#1}
  \def\testt{#2}
  \ifx\tst\empty
    \ifx\testt\empty
      \overline{S}_t
    \else
      \overline{S}_t\left( #2 \right)
    \fi
  \else
    \overline{S}_{ #1 }\left( #2 \right)
  \fi
}
\newcommand{\lV}{\underline{V}} % l(ower)V
\newcommand{\uV}{\overline{V}} % u(pper)V
\newcommand{\supp}[1]{\mathrm{supp}\left(#1\right)} % support of a random variable
\newcommand{\cU}{\mathcal{U}_t^w} % constraint multifunction

\newcommand{\dyn}[2][w]{ % dynamic
  \def\tst{#2}
  \ifx\tst\empty
    f_t^{#1}
  \else
    f_t^{#1}(#2)
  \fi
}
\newcommand{\costname}{c}
\newcommand{\cost}[2][w]{ % cost function
  \def\tst{#2}
  \ifx\tst\empty
    c_t^{#1}
  \else
    c_t^{#1}(#2)
  \fi
}
\newcommand{\pB}[3][w]{ % p(ointwise)B(ellman operator)
  \def\tst{#2}
  \def\tstt{#3}
  \ifx\tst\empty
    \mathcal{B}_t^{#1}
  \else
    \ifx\tstt\empty
      \mathcal{B}_t^{#1}\left( #2 \right)
    \else
      \mathcal{B}_t^{#1}\left( #2 \right) \left( #3 \right)
      \fi
  \fi
}
\newcommand{\aB}[2]{ % a(verage)B(ellman operator)
  \def\tst{#1}
  \def\testt{#2}
  \ifx\tst\empty
    \mathfrak{B}_t
  \else
    \ifx\testt\empty
      \mathfrak{B}_t\left( #1 \right)
    \else
      \mathfrak{B}_t\left( #1 \right) \left( #2 \right)
    \fi
  \fi
}

\newcommand{\OBSERVATION}{\mathbb{O}}
\newcommand{\observation}{o}
\newcommand{\Observation}{{O}}
\newcommand{\belief}{b}

\newcommand{\TransitionState}[4]{P_{#4}^{#1}\np{#2,#3}}
\newcommand{\TransitionStateName}{P}
\newcommand{\ObservationLaw}[4]{Q_{#4}\nsetp{#1}{#2,#3}}

\usepackage{a4wide}
\title{Tropical Dynamic Programming for \\ Lipschitz Multistage Stochastic Programming}

\author{Marianne Akian \thanks{INRIA Saclay Ile-de-France and CMAP, \'{E}cole Polytechnique, CNRS, France.}
  \and Jean-Philippe Chancelier\thanks{CERMICS, \'{E}cole des Ponts ParisTech, France.}
  \and Beno\^{i}t Tran\footnotemark[2]\,\,\footnotemark[1]
}

\begin{document}

\maketitle

\begin{abstract}
 We present an algorithm called Tropical Dynamic Programming (TDP) which builds upper and lower approximations of the Bellman value functions in risk-neutral Multistage Stochastic Programming (MSP), with independent noises of finite supports.

 To tackle the curse of dimensionality, popular parametric variants of Approximate Dynamic Programming approximate the Bellman value function as linear combinations of basis functions. Here, Tropical Dynamic Programming builds upper (resp. lower) approximations of a given value function as min-plus linear (resp. max-plus linear) combinations of "basic functions". At each iteration, TDP adds a new basic function to the current combination following a deterministic criterion introduced by Baucke, Downward and Zackeri in 2018 for a variant of Stochastic Dual Dynamic Programming.

 We prove, for every Lipschitz MSP, the asymptotic convergence of the generated approximating functions of TDP to the Bellman value functions on sets of interest. We illustrate this result on MSP with linear dynamics and polyhedral costs.
\end{abstract}

% post abstract

\section{Introduction}
In this article we study multistage stochastic optimal control problems in the
hazard-decision framework (hazard comes first, decision second). Starting from a
given state $x_0$, a decision maker observes the outcome $w_1$ of a
random variable $\mathbf{W_1}$, then decides on a control $u_0$ which induces a
\emph{known} cost $c_0^{w_1}\np{x_0, u_0}$ and the system evolves to a future
state $x_1$ from a \emph{known} dynamic: $x_1 = f_0^{w_1}\np{x_0, u_0}$. Having observed
a new random outcome, the decision maker makes a new decision based on
this observation which induces a known cost, then the system evolves to a known
future state, and so on until $T$ decisions have been made. At the last step, there
are constraints on the final state $x_T$ which are modeled by a final cost function
$\psi$. The decision maker aims to minimize the average cost of her decisions.

Multistage Stochastic optimization Problems (MSP) can be formally described by
the following optimization problem
\begin{equation}
 \label{MSP}
 \begin{aligned}
   & \min_{\np{\mathbf{X}, \mathbf{U}}} \E \left[ \sum_{t=0}^{T-1} \cost[\mathbf{W_{t+1}}]{\mathbf{X_t}, \mathbf{U_t}} + \psi\np{\mathbf{X_T}} \right], \\
   & \text{s.t.} \ \mathbf{X_0} = x_0 \ \text{given}, \forall t\in \ce{0,T-1},                                                                          \\
   & \mathbf{X_{t+1}} = \dyn[\mathbf{W_{t+1}}]{\mathbf{X_t}, \mathbf{U_t}},                                                                             \\
   & \sigma\np{\mathbf{U_t}} \subset \sigma\np{ \mathbf{W_1}, \ldots , \mathbf{W_{t+1}}},
 \end{aligned}
\end{equation}
where $\np{\mathbf{W_{t}}}_{t\in \ce{1,T}}$ is a given sequence of independent random variables each with values in some measurable set $(\mathbb{W}_t,\mathcal{W}_t)$. We refer to the random variable $\mathbf{W_{t+1}}$ as a \emph{noise} and throughout the remainder of the article we assume the following on the sequence of noises.
\begin{assumption}
 \label{whitenoise}
 Each random variable $\mathbf{W_t}$ in Problem \eqref{MSP} has finite support and the sequence of random variable $\np{\mathbf{W_{t}}}_{t\in \ce{1,T}}$ is independent.
\end{assumption}

One approach to solving MSP problems is by dynamic programming, see for example
\cite{Be2016,Ca.Ch.Co.De2015,Pf.Pi2014,Sh.De.Ru2009}. For some integers
$n,m \in \N$, denote by $\X = \R^n$ the \emph{state space} and $\U = \R^m$ the
\emph{control space}. Both $\X$ and $\U$ are endowed with their euclidean
structure and borelian structure. We define the pointwise Bellman operators
$\pB{}{}$ and the average Bellman operators $\aB{}{}$ for every
$t \in \ic{0,T-1}$.  For each possible realization $w\in \mathbb{W}_{t+1}$ of
the noise $\mathbf{W_{t+1}}$, for every function $\func : \X \to \Rb$ taking
extended real values in $\Rb = \RR \cup\na{\pm\infty}$, the function
$\pB{\func}{\cdot}: \X \to \Rb$ is defined by
\begin{equation*}
 \forall x\in \X,\eqsepv  \pB{\func}{x} = \min_{u\in \U} \Bp{\cost{x,u} + \func\bp{\dyn{x,u}}}
 \eqfinp
\end{equation*}
Now, the average Bellman operator $\aB{}{}$ is the mean of all the pointwise Bellman operators with respect to the probability law of
$\mathbf{W_{t+1}}$. That is, for every $\phi : \X \to \Rb$, we have that
\begin{equation*}
 \forall x \in \X
 \eqsepv \aB{\func}{x} = \E \bc{ \pB[\mathbf{W_{t+1}}]{\func}{x}} =
 \E\Bc{ \min_{u\in \U}\Bp{\cost[\mathbf{W_{t+1}}]{x,u} + \phi\bp{\dyn[\mathbf{W_{t+1}}]{x,u} }}}
 \eqfinp
\end{equation*}
The average Bellman operator can be seen as a one stage operator which computes
the value of applying the best (average) control at a given state $x$. Note
that in the hazard-decision framework assumed here, the control is taken after
observing the noise. Now, the Dynamic Programming approach states that in order to solve
MSP Problems~\eqref{MSP}, it suffices to solve the following system of
\emph{Bellman equations}~\eqref{BellmanEquations},
\begin{equation}
 \label{BellmanEquations}
 V_T = \psi \quad\text{and}\quad
 \forall t \in \ic{0,T-1}, V_t  = \aB{V_{t+1}}{}
 \eqfinp
\end{equation}
Solving the Bellman equations means computing recursively backward in time the
\emph{(Bellman) value functions} $V_t$. Finally, the value $V_0(x_0)$ is the solution of the multistage Problem~\ref{MSP}.

Grid-based approach to compute the value functions suffers from the so-called
curse of dimensionality. Assuming that the value functions
$\na{V_t}_{t\in \ic{0,T}}$ are convex, one approach to bypass this difficulty is
proposed by Pereira and Pinto \cite{Pe.Pi1991} with the Stochastic Dual Dynamic
Programming (SDDP) algorithm which computes piecewise affine approximations of
each value function $V_t$. At a given iteration $k\in \N^*$ of SDDP, for every
time step $t\in \ce{0,T}$, the value function $V_t$ is approximated by
$\lV_t^k = \max_{\func \in \lFunc_k} \func$ where $\lFunc_k$ is a finite set of
affine functions. Then, given a realization of the noise process
$\np{W_t}_{t\in \ce{1,T}}$, the decision maker computes an optimal trajectory
associated with the approximations $\np{\lV_t^k}_{t\in \ce{0,T}}$ and add a new
mapping, $\func_t^{k+1}$ (named cut) to the current collection $\lFunc_t^k$
which define $\lV_t^k$, that is
$\lFunc_t^{k+1} = \lFunc_t^k \cup \left\{ \func_t^{k+1} \right\}$. Although SDDP
does not involve discretization of the state space, one of its computational
bottleneck is the lack of efficient stopping criterion: SDDP easily builds lower
approximations of the value function but upper approximations are usually
computed through a costly Monte-Carlo scheme.

In order to build upper approximations of the value functions, Min-plus methods
were studied (\emph{e.g.} \cite{Mc2007,Qu2014}) for optimal control problems in
continuous time. When the value functions $\na{V_t}_{t\in \ic{0,T}}$ are convex
(or more generally, semiconcave), discrete time adaptations of Min-plus methods
build for each $t\in \ic{1,T}$ approximations of convex value function $V_t$ as
finite infima of convex quadratic forms. That is, at given iteration
$k\in \N$, we consider upper approximations defined as
$\uV_t^k = \min_{\func \in \uFunc_t^k} \func$, where
$\uFunc_t^k$ is a finite set of convex quadratic forms. Then, a sequence of
\emph{trial points} $\np{x_t^k}_{t\in \ce{0,T}}$ are drawn (\emph{e.g.}
uniformly on the unit sphere as in \cite{Qu2014}) and for every $t\in \ce{0,T{-}1}$
a new function $\func_t^{k+1}$ is added,
$\uFunc_t^{k+1} = \uFunc_t^k \cup \left\{ \func_t^{k+1}\right\}$. The function
$\func_t^{k+1}$ should be compatible with the Bellman equation, in particular it
should be \emph{tight}, \emph{i.e.} the Bellman equations should be satisfied at
the trial point,
\[
 \aB{\func_{t+1}^{k+1}}{x_t^k} = \func_t^{k+1}\np{x_t^k}.
\]
In~\cite{Ak.Ch.Tr2018}, the authors present a common framework for a
deterministic version of SDDP and a discrete time version of Min-plus
algorithms. Moreover, the authors give sufficient conditions on the way the trial
points have to be sampled in order to obtain asymptotic convergence of either
upper or lower approximations of the value functions. Under these conditions,
the main reason behind the convergence of these algorithm was shown to be that
the Bellman equations \eqref{BellmanEquations} are asymptotically satisfied on
all cluster points of possible trial points. In this article, we would like to
extend the work of \cite{Ak.Ch.Tr2018} by introducing a new algorithm called
Tropical Dynamic Programming (TDP).

In \cite{Ba.Do.Za2018,Ph.de.Fi2013}, is studied approximation schemes where lower
approximations are given as a suprema of affine functions and upper
approximations are given as a polyhedral function. We aim in this article to
extend, with TDP, the approach of \cite{Ba.Do.Za2018,Ph.de.Fi2013} considering more
generally that lower approximations are max-plus linear combinations of some
\emph{basic functions} and upper approximations are min-plus linear combinations
of other \emph{basic functions} where basic functions are defined later. TDP can
be seen as a tropical variant of parametric approximations used in Adaptive
Dynamic Programming (see \cite{Be2019,Po2011}) where the value functions are
approximated by linear combinations of basis functions. In this article, we
will:
\begin{enumerate}
 \item Extend the deterministic framework of \cite{Ak.Ch.Tr2018} to Lipschitz
       MSP defined in \Cref{MSP} and introduce TDP, see Section~\ref{sec:TDP}.
 \item Ensure that upper and lower approximations converge to the true value
       functions on a common set of points, see Section~\ref{sec:convergence}. The main
       result of Section~\ref{sec:convergence} generalizes to any min-plus/max-plus
       approximation scheme the result of \cite{Ba.Do.Za2018} which was stated for a
       variant of SDDP.
 \item Explicitly give several numerically efficient ways to build upper and
       lower approximations of the value functions, as min-plus and max-plus linear
       combinations of some simple functions, see Section~\ref{sec:numerique}.
       % \item Briefly explain how TDP could be used to solve Partially Observed Markov
       %       Decision Processes (POMDP), see Section~\ref{sec:POMDP}.
\end{enumerate}

\section{Tropical Dynamical Programming on Lipschitz MSP}
\label{sec:TDP}

\subsection{Lipschitz MSP with independent finite noises}

For every time step $t\in \ce{1,T}$, we denote by $\supp{\mathbf{W_{t}}}$ the
\emph{support} of the discrete random variable $\mathbf{W}_{t}$ \footnote{The
 support of the discrete random variable $\mathbf{W}_t$ is equal to
 the set $\left\{ w\in \mathbb{W}_t \mid \Pro\np{\mathbf{W_t}= w} > 0 \right\}$.}  and
for a given subset $X \subset \X$, we denote by $\pi_X$ the euclidean projector
on $X$. State and control constraints for each time $t$ are modeled in the cost
functions which may possibly take infinite values outside of some given
sets. Now, we introduce a sequence of sets
$\na{X_t}_{t \in \ic{0,T}}$ which only depend on the problem data and make the following compactness assumption:
\begin{assumption}[Compact state space]
 \label{compactStates}
 For every time $t \in \ce{0,T}$, we assume that the set $X_t$ is a nonempty compact set in $\X$ where
 the sequence of sets $\na{X_t}_{t \in \ic{0,T}}$ is defined, for all $t\in \ce{0,T-1}$, by
 \begin{equation}
  \label{Xt}
  X_t := \bigcap_{w\in \supp{\mathbf{W_{t+1}}}}  \pi_{\X}\np{ \dom\, \cost{}{} }
  \eqfinv
 \end{equation}
 and for $t=T$ by $X_T = \dom \psi$.
\end{assumption}

For each noise $w\in \supp{\mathbf{W_{t+1}}}$, $t\in \ce{0,T-1}$, we also introduce the \emph{constraint set-valued mapping} $\cU : \X \rightrightarrows \U$ defined for every $x \in \X$ by
\begin{equation}
 \label{constraintset-valued mapping}
 \cU\np{x} := \bset{ u \in \U}{ \cost{x,u} < +\infty \ \text{and} \ \dyn{x,u} \in X_{t+1}}
 \eqfinp
\end{equation}

We will assume that the data of Problem~\eqref{MSP} is Lipschitz in the sense defined below. Let us stress that we do not assume structure on the dynamics or costs like linearity or convexity, only that they are Lipschitz.

\begin{assumption}[Lipschitz MSP]
 \label{LipschitzMSP}
 For every time $t\in \ce{0,T-1}$, we assume that for each
 $w\in \supp{\mathbf{W_{t+1}}}$, the dynamic $\dyn{}{}$, the cost $\cost{}{}$
 are Lipschitz continuous on $\dom\, \cost{}{}$ and the set-valued mapping
 constraint $\cU$ is Lipschitz continuous on $X_t$, \emph{i.e.} for some
 constant $L_{\mathcal{U}_t^w} >0$, for every $x_1, x_2 \in X_t$, we have
 \begin{equation}
  d_{\mathcal{H}}\bp{\cU\np{x_1}, \cU\np{x_2}} \leq L_{\cU} \lVert x_1 - x_2 \rVert.\footnotemark
 \end{equation}
\end{assumption}
Computing a (sharp) Lipschitz constant for the set-valued mapping
$\cU : \X \rightrightarrows \U$ is difficult. However, when the graph of the
set-valued mapping $\cU$ is polyhedral, as in the linear-polyhedral framework
studied in \Cref{sec:numerique}, one can compute a Lipschitz constant for $\cU$.
We make the following assumption in order to ensure that the domains of the
value functions $V_t$ are chosen by the decision maker. It can be seen as a
recourse assumption.  \footnotetext{The Hausdorff distance $d_{\mathcal{H}}$
 between two nonempty compact sets $X_1, X_2$ in $\X$ is defined by
 \[
  d_{\mathcal{H}}\np{X_1, X_2} = \max\np{ \max_{x_1 \in X_1} d\np{x_1, X_2}, \max_{x_2\in X_2} d\np{X_1,x_2}} = \max\np{\max_{x_1 \in X_1} \min_{x_2\in X_2} d\np{x_1,x_2}, \max_{x_2\in X_2} \min_{x_1\in X_1} d\np{x_1,x_2}}.
 \]}

\begin{assumption}[Recourse assumption]
 \label{recourse}
 Given $t\in \ce{0,T{-}1}$, for every noise realization
 $w \in \supp{\mathbf{W_t}}$ the set-valued mapping
 $\cU : \X \rightrightarrows \U$ defined in \eqref{constraintset-valued
  mapping} is nonempty compact valued.
\end{assumption}

\emph{A priori}, it might be difficult to compute the domain of each value
function $V_t$. However, under the recourse \Cref{recourse}, we have that
$\dom V_t := X_t$ and thus the domain of each value function is known to the
decision maker.

\begin{lemma}[Known domains of $V_t$]
 Under Assumptions~\ref{whitenoise} and~\ref{recourse}, for every
 $t\in \ce{0,T}$, the domain of $V_t$ is equal to $X_t$.
\end{lemma}

\begin{proof}
 We make the proof by backward induction on time.  At time $t=T$, we have
 $V_T = \psi$ and thus $\dom V_T = \dom \psi = X_T$. Now, for a given
 $t\in \ce{0,T{-}1}$, we assume that $\dom V_{t+1} = X_{t+1}$ and we
 prove that $\dom V_t = X_t$.

 First, fix $x \in X_t$. Then, for every $w\in \supp{\mathbf{W_{t+1}}}$,
 using~\Cref{recourse}, $\cU\np{x}$ is nonempty and thus $V_t\np{x} <
  +\infty$. Moreover, by Assumptions~\ref{LipschitzMSP} and
 Assumptions~\ref{recourse} the optimization problem
 \[
  \min_{u\in \U}\Bp{\cost{x,u} + V_{t+1}\bp{\dyn{x,u}}} = \min_{u\in \cU\np{x}} \Bp{\cost{x,u} + V_{t+1}\bp{\dyn{x,u}}},
 \]
 consists in the minimization of a continuous function in $u$ over a nonempty compact set. Denote by $u^w \in \cU\np{x}$ a minimizer of this optimization problem. We have, denoting by
 $\na{p_w}_{w\in \supp{\mathbf{W_{t+1}}}}$ the discrete probability law of the random variable
 $\mathbf{W_{t+1}}$, that
 \begin{align}
  V_t\np{x}
   & = \aB{V_{t+1}}{x}                \notag
  \\
   & = \E\bc{ \pB[\mathbf{W_{t+1}}]{V_{t+1}}{x}  \notag }
  \\
   & = \sum_{w\in \supp{\mathbf{W_{t+1}}}} p_w
  \inf_{u\in \U} \Bp{\cost{x,u} + V_{t+1}\bp{\dyn{x,u}}} \notag
  \\
   & = \sum_{w\in \supp{\mathbf{W_{t+1}}}} p_w \Bp{\cost{x,u^w} + V_{t+1}\bp{\dyn{x,u^w}}}
  \notag \eqfinp
 \end{align}
 As every term in the right hand side of the previous equation is finite, we
 have $V_t\np{x} < + \infty$ and thus $x\in \dom V_t$.

 Second, fix $x\notin X_t$. Then, there exists an element
 $w \in \supp{\mathbf{W_{t+1}}}$ such that $\cost{x,u} = +\infty$ for every
 control $u \in \U$.  We therefore have that $V_t\np{x} = +\infty$ and
 $x \not\in  \dom V_t$.

 We conclude that $\dom V_t = X_t$ which ends the proof.
\end{proof}

In \Cref{sec:numerique}, it will be crucial for numerical efficiency to have a
good estimation of the Lipschitz constant of the function
$\aB{V_{t+1}^k}{}$.

We now prove that under Assumptions~\ref{LipschitzMSP} and
Assumptions~\ref{recourse}, the operators $\aB{}{}$ preserve Lipschitz
regularity. Given a $L_{t+1}$-Lipschitz function $\func$ and
$w\in \supp{\mathbf{W_{t+1}}}$, in order to compute a Lipschitz constant of the
function $\pB{\func}{\cdot}$ we exploit the fact that the set-valued constraint
mapping $\cU$ and the data of Problem \ref{MSP} are Lipschitz in the sense of
Assumptions~\ref{LipschitzMSP}.  This was mostly already done in
\cite{Ak.Ch.Tr2018}, but for the sake of completeness, we will slightly adapt
its statement and proof.

\begin{proposition}[$\aB{}{}$ is Lipschitz regular]
 \label{lipschitz_regularity}
 Let $\func : \X \to \Rb$ be given. Under Assumptions~\ref{whitenoise} to
 \ref{recourse}, if for some $L_{t+1}> 0$, $\func$ is $L_{t+1}$-Lipschitz on
 $X_{t+1}$, then the function $\aB{\func}{}$ is $L_t$-Lipschitz on $X_t$ for
 some constant $L_t > 0$ which only depends on the data of Problem~\ref{MSP} and
 $L_{t+1}$.
\end{proposition}

\begin{proof}
 Let $\func : \X \to \Rb$ be a $L_{t+1}$-Lipschitz function on $X_{t+1}$. We
 will show that for each $w\in \supp{\mathbf{W_{t+1}}}$, the mapping
 $\pB{\func}{\cdot}$ is $L_w$-Lipschitz for some constant $L_w$ which only
 depends on the data of problem \eqref{MSP}. Fix
 $w \in \supp{\mathbf{W_{t+1}}}$ and $x_1, x_2 \in X_t$. Denote by $u_2^*$ an
 optimal control at $x_2$ and $w$, that is
 $u_2^* \in \argmin_{u \in \cU\np{x_2}}\Bp{ \cost{x_2,u} + \func\bp{\dyn{x_2,u}}}$,
 or equivalently, $u_2^*$ satisfies
 \begin{equation}
  \cost{x_2,u_2^*} + \func\bp{\dyn{x_2,u_2^*}} = \pB{\func}{x_2}.
  \label{eq:u2-optimal}
 \end{equation}
 Then, for every $u_1\in \cU(x_1)$ we successively have
 \begin{align*}
  \pB{\func}{x_1}
   & \le \cost{x_1,u_1} + \func\bp{\dyn{x_1,u_1}}
  \tag{as $u_1\in \cU(x_1)$ is admissible}
  \\
   & \leq \pB{\func}{x_2} +  \cost{x_1,u_1} + \func\bp{\dyn{x_1,u_1}}  - \pB{\func}{x_2}
  \\
   & = \pB{\func}{x_2}  + \bp{\cost{x_1,u_1} - \cost{x_2,u_2^*}} +
  \Bp{\func\bp{\dyn{x_1, u_1}} - \func\bp{\dyn{x_2, u_2^*}}}
  \tag{using~\eqref{eq:u2-optimal}}
  \\
   & \leq \pB{\func}{x_2} + L\bp{ \bnorm{x_1 - x_2} + \bnorm{u_1 - u_2^*}},
  \tag{by \Cref{LipschitzMSP}} \label{eq:lipschitz_regularity1}
 \end{align*}
 where $L = \max\np{L_{\cost{}{}}, L_{t+1}L_{\dyn{}{}}}$.
 Now, as the set-valued mapping $\cU$ is $L_{\cU}$-Lipschitz, there exists
 $\tilde{u}_1 \in \cU\np{x_1}$ such that
 \[
  \lVert \tilde{u}_1 - u_2^*\rVert \leq L_{\cU} \lVert x_1 - x_2 \rVert.
 \]
 Hence, setting $L_w := \max\np{L_{\cost{}{}}, L_{t+1}L_{f_t^w}} \np{1 + L_{\cU}}$,
 we obtain
 \[
  \pB{\func}{x_1} - \pB{\func}{x_2} \leq L_t \lVert x_1 - x_2 \rVert.
 \]
 Reverting the role of $x_1$ and $x_2$ we get the converse inequality.  Hence,
 we have shown that, for every $w \in \supp{\mathbf{W_{t+1}}}$, the mapping
 $\pB{\func}{}$ is $L_w$-Lipschitz. Thus, setting $L_t = \np{\sum_{w}p_w L_w}$,
 we have
 \begin{align*}
  \big\lvert  \aB{\func}{x_1} - \aB{\func}{x_2} \big\rvert
   & \leq \sum_{w\in \supp{\mathbf{W_{t+1}}}}
  p_w \big\lvert \pB{\func}{x_1} - \pB{\func}{x_2} \big\rvert
  \\
   & \leq \Bp{\sum_{w\in \supp{\mathbf{W_{t+1}}}} p_w L_w }
  \norm{x_1 - x_2}\eqfinv
 \end{align*}
 as $\pB{\func}{}$ is $L_w$-Lipschitz.
 We obtain that the mapping $\aB{\func}{}$ is $L_t$-Lipschitz continuous
 on $\dom V_t$ and this concludes the proof.
\end{proof}

The explicit constant $L_t$ computed in the proof of
Proposition~\ref{lipschitz_regularity} does not exploit any possible structure
of the data, \emph{e.g.} linearity. In the presence of such structure or
possible decomposition, it is possible to greatly reduce the value of the $L_t$
constant. However, in the sequel, we only care for the regularity result given
in Proposition~\ref{lipschitz_regularity} and computing sharper bounds under
some specific structure is left for future works.

Using the fact that the final cost function $\psi = V_T$ is Lipschitz on $X_T$, by successive applications of
Proposition~\ref{lipschitz_regularity}, one gets the following corollary.

\begin{corollary}[The value functions of a Lipschitz MSP are Lipschitz continuous]
 \label{Lipschitz_Vt}
 For every time step $t\in \ce{0,T}$, the value function $V_t$ is
 $L_{V_t}$-Lipschitz continuous on $X_t$ where $L_{V_t} > 0$ is a constant which
 only depends on the data of Problem~\ref{MSP}.
\end{corollary}

\subsection{Tight and valid selection functions}
We formally define now what we call \emph{basic functions}. In the sequel, the
notation in bold $\Funcb_t$ will stand for a set of basic functions and
$\Func_t$ will stand for a subset of $\Funcb_t$.
\begin{definition}[Basic functions]
 Given $t\in \ce{0,T}$, a \emph{basic function} $\func : \X \to \Rb$ is a
 $L_{V_t}$-Lipschitz continuous function on $X_t$, where the constant
 $L_{V_t}>0$ is defined in Corollary~\ref{Lipschitz_Vt}.
\end{definition}

In order to ensure the convergence of the scheme detailed in the introduction,
at each iteration of TDP algorithm a basic functions which is be \emph{tight}
and \emph{valid} in the sense below is added to the current sets of basic
functions.  The idea behind these assumptions is to ensure that the Bellman
equations \eqref{BellmanEquations} will gradually be satisfied: it is too
numerically hard to find functions satisfying the Bellman equations
\eqref{BellmanEquations}, however tightness and validity can be checked
efficiently and this will be enough to ensure asymptotic convergence of our TDP
algorithm.

There is a dissymmetry for the validity assumption which depends on whether the
decision maker wants to build upper or lower approximations of the value
functions. In \S\ref{subsec:TDP}, we will assume that the decision maker has, at
hist disposal, two sequences of selection functions
$\np{\uSelection[]{}}_{t\in \ce{0,T}}$ and
$\np{\lSelection[]{}}_{t\in \ce{0,T}}$. The former to select basic functions for the
upper approximations and the latter for the lower approximations of $V_t$. We write
$\ZSelection[]{}$ when designing either $\uSelection[]{}$ or $\lSelection[]{}$
and denote by $\uvopt_{\uFunc_t}$ (resp. $\lvopt_{\lFunc_t}$) the pointwise
infimum (resp. pointwise supremum) of basic functions in $\lFunc_t$ (resp. in
$\uFunc_t$) when approximating from above (resp. below) a maping $V_t$. The
Figure~\ref{fig:selection_u_sddp} illustrates the formal definition of selection
functions given below. Given a set $Z$, we denote by $\mathcal{P}\np{Z}$ its
\emph{power set}, \emph{i.e.} the set of all subsets included in $Z$.

\begin{definition}[Selection functions]
 \label{CompatibleSelection}
 Let a time step $t \in \ce{0,T-1}$ be fixed. A \emph{selection function} or
 simply \emph{selection function} is a mapping $\ZSelection[]{}$ from
 $\mathcal{P}\np{\Funcb_{t+1}}{\times}X_t$ to $\Funcb_t$ satisfying the
 following properties
 \begin{itemize}
  \item \textbf{Tightness}: for every set of basic functions
        $\Func_{t+1} \subset \Funcb_{t+1}$ and $x\in X_t$, the mappings
        $\ZSelection{\Func_{t+1}, x}$ and $\pB{V_{\Func_{t+1}}}{\cdot}$ coincide at
        point $x$, that is
        \[
         \ZSelection{\Func_{t+1}, x}\np{x} = \aB{\vopt_{\Func_{t+1}}}{x}.
        \]
  \item \textbf{Validity}: for every set of basic functions
        $\Func_{t+1} \subset \Funcb_{t+1}$ and for every $x\in X_t$ we have
        \begin{align}
         \uSelection{\Func_{t+1}, x} \geq \aB{\vopt_{\Func_{t+1}}}{\cdot}, \tag{when building upper approximations} \\
         \lSelection{\Func_{t+1}, x} \leq \aB{\vopt_{\Func_{t+1}}}{\cdot}. \tag{when building lower approximations}
        \end{align}
 \end{itemize}
 For $t= T$, we also say that $S_T: X_T \to \Funcb_T$ is a \emph{selection function}
 if the mapping $S_T$ is \emph{tight} and \emph{valid} with a modified definition of
 \emph{tight} and \emph{valid} defined now. The mapping $S_T$ is said to be
 \emph{valid} if, for every $x\in X_T$, the function $S_T\left( x \right)$
 remains above (resp. below) the value function at time $T$ when building upper
 approximations (resp.  lower approximations). The mapping $S_T$ is said to be
 \emph{tight} if it coincides with the value function at point $x$, that is for
 every $x\in X_T$ we have
 \[
  S_T\left( x\right)\left(x\right) = V_T\np{x}.
 \]
\end{definition}

\begin{remark}
 Note that the validity and tightness assumptions at time $t=T$ is stronger than at times $t<T$ as the final cost function is a known data, we are allowed to enforce conditions directly on the value function $V_T$ and not just the on the image of the current approximations at time $t+1$ as it is the case when $t<T$.
\end{remark}

\begin{figure}
 \centering
 \includegraphics[width=\linewidth]{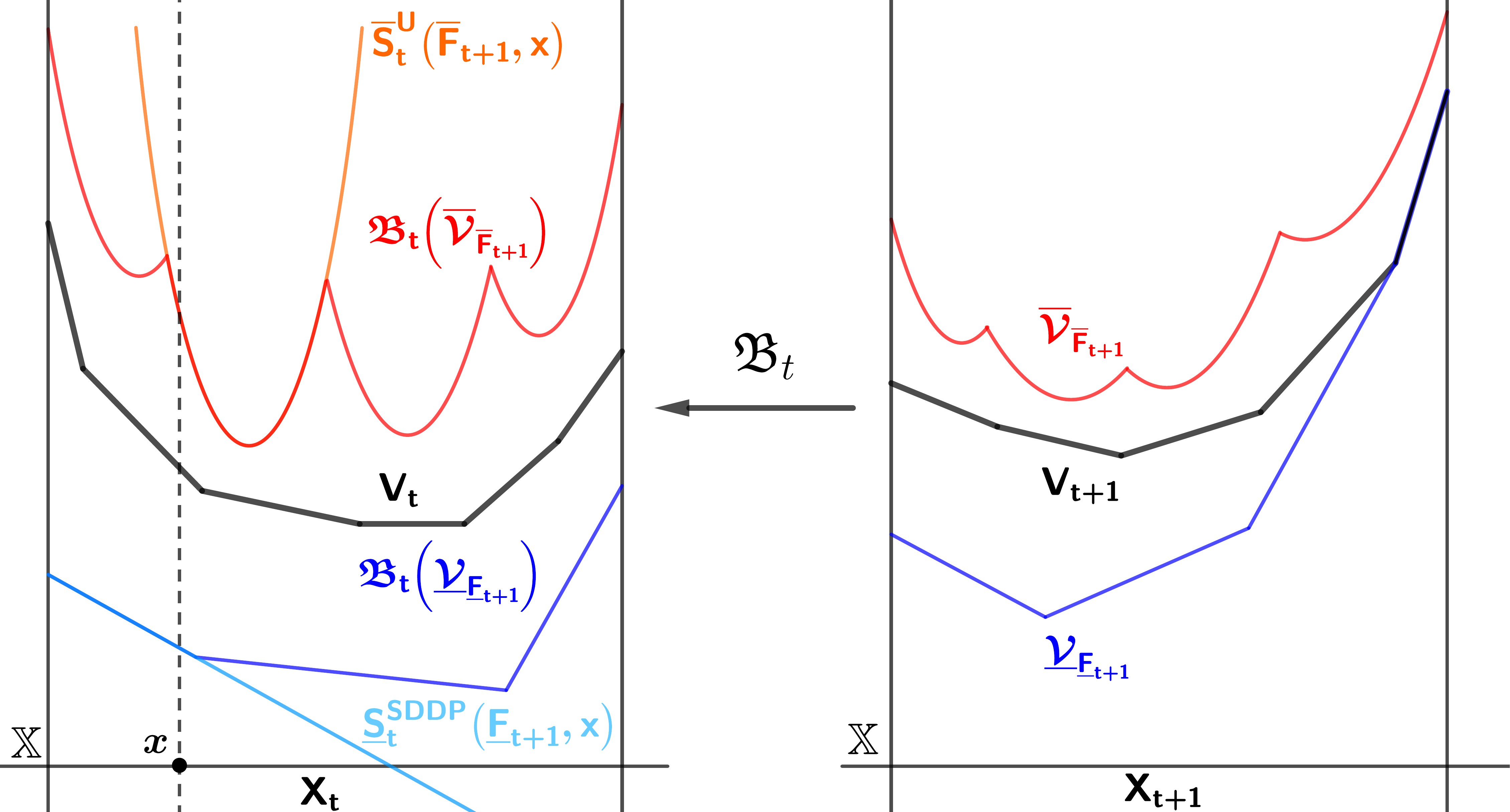}
 \caption{\label{fig:selection_u_sddp} Given a time step $t\in \ce{0,T-1}$, we illustrate the notions of tightness and validity of selection functions. A selection function takes as input a trial point $x$ in the domain $X_t$ of $V_t$ and  a set of basic functions $\Func_t \subset \Funcb_{t+1}$ building the approximations at the future time step $t+1$ (right: pointwise suprema or infima of the basic functions). Then, the Bellman operator $\mathfrak{B}_t$ translates one step backward in time the right picture to the picture on the left. \\
  Tightness of the selection function enforces that the output is a function equal to the Bellman image of the future approximation of $V_{t+1}$ at $x$;
  it is a local property. \\
  Validity enforces that the output of the selection function remains below, or above, the Bellman image the approximation of $V_{t+1}$ everywhere on the domain of $V_t$; it is a global property. More details on these examples of selection functions in Section~\ref{sec:numerique}.}
\end{figure}

\subsection{The problem-child trajectory}

From the previous section, given a set of basic functions and a point in $\X$, a
selection function is used to computes a new basic function. We explain in this section
the algorithm used to select the points which are used for searching new basic functions.

In this section we present how to build a trajectory of states, without
discretization of the whole state space. Selection functions for both upper and
lower approximations of $V_t$ will be evaluated along it. This trajectory of
states, coined \emph{problem-child} trajectory, was introduced by Baucke,
Downward and Zackeri in 2018 (see \cite{Ba.Do.Za2018}) for a variant of SDDP first
studied by Philpott, de Matos and Finardi in 2013 (see \cite{Ph.de.Fi2013}).

We present in Algorithm~\ref{PC} a generalized problem-child trajectory, it is
the sequence of states on which we evaluate selection functions.

\begin{algorithm}
 \caption{Problem-child trajectory}
 \label{PC}
 \begin{algorithmic}
  \REQUIRE{Two sequences of functions from $\X$ to $\Rb$,
   $\ufunc_0,\ldots, \ufunc_T$ and $\lfunc_0,\ldots, \lfunc_T$ with respective
   domains equal to $\dom V_t$.}  \ENSURE{A sequence of states
   $\np{x_0^*, \ldots, x_T^*}$.}  \STATE{Set $x_0^* := x_0$.}
  \FOR{$t \in \ce{0,T{-}1}$} \FOR{$w\in \supp{\mathbf{W_{t+1}}}$} \STATE{Compute
   an optimal control $u_t^w$ for $\lfunc_{t+1}$ at $x_t^*$ for the given $w$
   \begin{equation}
    \label{eq:pc_control}
    u_t^w \in \argmin_{u\in \U} \Bp{ \cost{x_t^*,u} + \lfunc_{t+1}\bp{\dyn{x_t^*,u}} }.
   \end{equation}}
  \ENDFOR
  \STATE{Compute ``the worst'' noise $w^*\in \supp{\mathbf{W_{t+1}}}$.
   \emph{i.e.} the one which maximizes the ``future'' gap
   \[
    w^* \in \argmax_{w \in \supp{\mathbf{W_{t+1}}}} \bp{ \ufunc_{t+1} -\lfunc_{t+1} }\bp{\dyn{x_t^*,u_t^w}}.
   \]}
  \STATE{Compute the next state dynamics for noise $w^*$ and associated optimal control
  $u_t^{w^*}$:
  $$x_{t+1}^* = \dyn[w^*]{x_t^*,u_t^{w^*}}\eqfinp$$}
  \ENDFOR
 \end{algorithmic}
\end{algorithm}

One can interpret the problem child trajectory as the worst (for the noises)
optimal trajectory (for the controls) of the lower approximations. It is worth
mentioning that the problem-child trajectory is deterministic. The approximations of
the value functions will be refined along the problem-child trajectory only,
thus avoiding a discretization of the state space. The main computational
drawback of such approach is the need to solve Problem \eqref{eq:pc_control}
$\lvert \supp{\mathbf{W_{1}}} \rvert \cdot \ldots \cdot \lvert
 \supp{\mathbf{W_{T}}} \rvert$ times. Except on special instances like the
linear-quadratic case, one cannot expect to find a closed form expression for
solutions of Equation~\eqref{eq:pc_control}. However,
we will see in \Cref{sec:numerique} examples
where Problem \eqref{eq:pc_control} can be solved by Linear Programming or
Quadratic Programming.  Simply put, if one can solve efficiently the
deterministic problem \eqref{eq:pc_control} and if at each time step the set
$\supp{\mathbf{W_{t}}}$ remains of small cardinality, then
using the problem-child trajectory and the Tropical Dynamical
Algorithm presented below in \Cref{subsec:TDP}, one can solve MSP problems with
finite independent noises efficiently. This might be an interesting framework in
practice if at each step the decision maker has a few different forecasts on
which her inputs are significantly different.

\subsection{Tropical Dynamic Programming}
\label{subsec:TDP}

\begin{algorithm}[H]
 \caption{Tropical Dynamic Programming (TDP)}
 \label{TDP}
 \begin{algorithmic}
  \REQUIRE{For every $t\in \ce{0,T}$, two compatible selection functions
   $\uSelection[]{}$ and $\lSelection[]{}$. A sequence of independent random
   variables $\np{\mathbf{W_t}}_{t\in \ce{0,T-1}}$, each with finite support.}
  \ENSURE{For every $t\in \ce{0,T}$, two sequence of sets
   $\np{\uFunc_t^k}_{k\in\N}$, $\np{\lFunc_t^k}_{k\in \N}$ and the associated
   functions $\uV_t^k = \inf_{\func \in \uFunc_t^k} \func$ and
   $\lV_t^k = \sup_{\func \in \lFunc_t^k} \func$.}
  \STATE{Define for every
   $t\in \ce{0,T}$, $\uFunc_t^0 := \emptyset$ and $\lFunc_t^0 := \emptyset$.}
  \FOR{$k\geq 0$}
  \STATE{\emph{Forward phase}}
  \STATE{Compute the problem-child
  trajectory $\np{x_t^k}_{t\in \ce{0,T}}$ for the sequences
  $\np{\uvopt_{\uFunc_t^k}}_{t\in \ce{0,T}}$ and
  $\np{\lvopt_{\lFunc_t^k}}_{t\in \ce{0,T}}$ using Algorithm~\ref{PC}.}
  \STATE{\emph{Backward phase}}
  \STATE{At $t = T$, compute new basic functions
  $\ufunc_T := \overline{S}_T \left( x_T^{k}\right)$ and
  $\lfunc_T := \underline{S}_T \np{ x_T^{k}}$.}
  \STATE{Add them to current collections,
   $\uFunc_T^{k+1} := \uFunc_T^{k} \cup \left\{ \ufunc_T \right\}$ and
   $\lFunc_T^{k+1} := \lFunc_T^{k} \cup \left\{ \lfunc_T \right\}$.}
  \FOR{$t$
   from $T{-}1$ to $0$}
  \STATE{Compute new basic functions:
  $\ufunc_t := \uSelection{\uFunc^{k+1}_{t+1}, x_t^{k}}$ and
  $\lfunc_t := \lSelection{\lFunc^{k+1}_{t+1}, x_t^{k}}$.}  \STATE{Add them
   to the current collections:
   $\uFunc_t^{k+1} := \uFunc_t^{k} \cup \left\{ \ufunc_t \right\}$ and
   $\lFunc_t^{k+1} := \lFunc_t^{k} \cup \left\{ \lfunc_t \right\}$.}  \ENDFOR
  \ENDFOR
 \end{algorithmic}
\end{algorithm}

\section{Asymptotic convergence of TDP along the problem-child trajectory}
\label{sec:convergence}
% \begin{lemma}
%  \label{lemma:unifconv}
%  For some constant $L > 0$, let $\np{\func^k}_{k\in \N}$ be a bounded monotonic sequence of $L$-Lipschitz functions on $X_t$. Then, $\np{\func^k}_{k\in \N}$ converges uniformly on every compact set included in $X_t$ to a function $\func^*$ wich is $L$-Lipschitz continuous as well.
% \end{lemma}
% \begin{proof}
%  First, as the sequence of $L$-Lipschitz functions $\np{\func^k}_{k\in \N}$ is monotonic and bounded, it converges pointwise to a function $\func^*$ which is $L$-Lipschitz continuous on $X_t$ as well.
%  Now, as the sequence $\np{\func^k}_{k\in \N}$ is bounded, there exists $\phi_1$ and $\phi_2$ such that for every $k\in \N$, $\phi_1 \leq \func^k \leq \phi_2$. Note that as $\phi_1$ and $\phi_2$ are $L_t$-Lipschitz continuous, they are finite on $X_t$. Moreover, given a compact set $K\subset X_t$, we have that
%  \[
%   \lVert \phi_1 \rVert_{\infty, K} \leq \sup_{k\in \N} \lVert \func^k \rVert_{\infty, K} \leq \lVert \phi_2 \rVert_{\infty, K} .
%  \]
%  Hence, by Arzela-Ascoli~theorem the monotonic sequence $\np{\func^k}_{k\in \N}$ converges uniformly on $K$ to $\func^*$.
% \end{proof}

In this section, we will assume that Assumptions \eqref{whitenoise} to
\eqref{recourse} are satisfied. We recall that, under Assumption~\ref{recourse},
the sequence of sets $\na{X_t}_{t\in \ic{0,T}}$ defined in Equation~\eqref{Xt}
is known and for all $t\in \ce{0,T}$ the domain of $V_t$ is equal to $X_t$.
We denote by $\np{x_t^k}_{k\in \N}$ the sequence of trial points generated by TDP algorithm at time $t$ for
every $t\in \ce{0,T}$, and by $\np{u_t^k}_{k\in \N}$ and $\np{w_t^k}_{k\in \N}$
the optimal control and worst noises sequences associated for each time $t$
with $x_t^k$ in the problem-child trajectory in Algorithm~\ref{PC}.

Now, observe that for every $t\in \ce{0,T}$, the approximations of $V_t$ generated by TDP, $\np{\uV_t^k}_{k\in \N}$ and $\np{\lV_t^k}_{k \in \N}$, are respectively non increasing and non decreasing. Moreover, for every index $k\in \N$ we have
\[
 \lV_t^k \leq V_t \leq \uV_t^k.
\]
We refer to \cite[Lemma 7]{Ak.Ch.Tr2018} for a proof. Observing that the basic functions are all $L_{V_t}$-Lipschitz continuous on $X_t$ one can prove using Arzel\`a-Ascoli Theorem the following proposition.

\begin{proposition}[Existence of an approximating limit]
 \label{ExistenceLimits}
 Let $t\in \ce{0,T}$ be fixed, the sequences of
 functions $\left( \lV_t^k \right)_{k\in \N}$ and $\left( \uV_t^k \right)_{k\in \N}$ generated by Algorithm~\ref{TDP} converge
 uniformly on $X_t$ to two functions $\lV_t^*$ and $\uV_t^*$. Moreover, $\lV_t^*$ and $\uV_t^*$ are $L_{V_t}$-Lipschitz continuous on $X_t$ and satisfy
 \(
 \lV_t^* \leq V_t \leq \uV_t^* \;.
 \)
\end{proposition}
\begin{proof}
 Omitted as it is slight rewriting of \cite[Proposition 9]{Ak.Ch.Tr2018}.
\end{proof}

If we extract a converging subsequence of trial points, then using compactness,
extracting a subsubsequence if needed, one can find a find a subsequence of
trial points, and associated controls that jointly converge.
\begin{lemma}
 \label{lem:conv_commune}
 Fix $t \in \ce{0,T-1}$ and denote by $\np{x_t^k}_{k\in \N}$ the sequence of
 trial points generated by Algorithm~\ref{TDP} and by $\np{u_t^k}_{k\in \N}$ the
 sequence of associated optimal controls. There exists an increasing function
 $\sigma : \N \to \N$ and a state-control ordered pair
 $\np{x_t^*, u_t^*} \in X_t{\times}\U$ such that
 \begin{equation}
  \left\{
  \begin{aligned}
    & x_t^{\sigma(k)} \underset{k\to +\infty}{\longrightarrow} x_t^*, \\
    & u_t^{\sigma(k)} \underset{k\to +\infty}{\longrightarrow} u_t^*.
  \end{aligned}
  \right.
 \end{equation}
\end{lemma}
\begin{proof}
 Fix a time step $t\in \ce{0,T{-}1}$. First, by construction of the problem-child
 trajectories, the sequence $\np{x_t^k}_{k\in \N}$ remains in the subset
 $X_t$ that is $x_t^k \in X_t$ for all $k\in \NN$.

 Second, we show that the sequence of controls $\np{u_t^k}_{k\in \N}$ is
 included in a compact subset of $\U$.  Under Assumption~\ref{compactStates},
 $X_t$ is a nonempty compact subset of $\X$. For every
 $w\in \supp{\mathbf{W_{t+1}}}$ the set-valued mapping $\cU$ is Lipschitz
 continuous on $X_t$ under Assumption~\ref{LipschitzMSP}, hence upper
 semicontinuous on $X_t$.\footnote{The compact valued set-valued mapping
  $\cU : \X \rightrightarrows \U$ is \emph{upper semicontinuous} on $X_t$ if,
  for all $x_t \in X_t$, if an open set $U \subset \U$ contains $\cU\np{x_t}$
  then $\left\{ x \in \X \mid \cU{x} \subset U \right\}$ contains a
  neighborhood of $x_t$.}  Moreover, under recourse Assumption~\ref{recourse},
 $\cU$ is nonempty compact valued. Thus, by \cite[Proposition 11
  p.112]{Au.Ek1984}, its image $\cU\np{X_t}$ of the compact $X_t$ is a nonempty
 compact subset of $\U$. Finally as the random variable $\mathbf{W_{t+1}}$ has
 a finite support under Assumption~\ref{whitenoise}, the set
 $U_t := \cup_{w \in \supp{\mathbf{W_{t+1}}}} \, \cU\np{X_t}$ is a compact
 subset of $\U$. The sequence $\np{u_t^k}_{k\in \N}$ remains in $U_t$ and
 therefore we conclude that it remains in a compact subset of $\U$.

 Finally, as the sequence $\np{x_t^k, u_t^k}_{k\in \N}$ is included in the compact subset $X_t{\times}U_t$
 of $\X{\times}\U$, one can extract a converging subsequence, hence the result.
\end{proof}

Lastly, we will use the following elementary lemma, whose proof is omitted.
\begin{lemma}
 \label{lem:unif_conv}
 Let $\np{g^k}_{k\in \N}$ be a sequence of functions that converges uniformly on
 a compact $K$ to a function $g^*$. If $\np{y^k}_{k\in \N}$ is a sequence of
 points in $K$ that converges to $y^* \in K$ then one has
 \[
  g^k\np{y^k} \underset{k\rightarrow + \infty}{\longrightarrow} g^*\np{y^*}.
 \]
\end{lemma}

We now state the main result of this article. For a fixed $t\in \ce{0,T}$,
as the Bellman value function $V_t$ is always sandwiched between the sequences of
upper and lower approximations, if the gap between upper and lower approximations
vanishes at a given state value $x$, then upper and lower approximations will both
converge to $V_t(x)$.
Note that, even though a MSP is a stochastic optimization problem, the
convergence result below is not. Indeed, we have assumed (see \Cref{whitenoise}) that
the noises have finite supports, thus under careful selection of scenario as done
by the Problem-child trajectory, we get a ``sure'' convergence.

\begin{theorem}[Vanishing gap along problem-child trajectories]
 Denote by $\np{\overline{V}_t^k}_{k\in \N}$ and $\np{\underline{V}_t^k}_{k\in \N}$ the
 approximations generated by the Tropical Dynamic Programming algorithm. For
 every $k\in \N$ denote by $\np{x_t^k}_{0 \leq t \leq T}$ the current
 Problem-child trajectory.

 Then, under Assumptions~\ref{whitenoise} to~\ref{recourse},
 we have that
 \[
  \uV_t^k\np{x_t^k} - \lV_t^k\np{x_t^k} \underset{k\to +\infty}{\longrightarrow 0}
  \quad\text{and}\quad \overline{V}_t^*\np{x_t^*} = \underline{V}_t^*\np{x_t^*}
  \eqfinv
 \]
 for every accumulation point $x_t^*$ of the sequence $\np{x_t^k}_{k\in \N}$.
\end{theorem}

\begin{proof}
 We prove by backward recursion that, for every $t\in \ce{0,T}$,
 for every accumulation point $x_t^*$ of the sequence $\np{x_t^k}_{k\in \N}$, we have
 \begin{equation}
  \label{ConvergingBonds}
  \overline{V}_t^*\np{x_t^*} = \underline{V}_t^*\np{x_t^*}.
 \end{equation}
 By a direct consequence of the tightness of the selection functions one has
 that for every $k\in \N$,
 $\overline{V}_T^k\np{x_T^k} = V_T\np{x_T^k} =
  \underline{V}_T^k\np{x_T^k}$. Thus, the equality~\eqref{ConvergingBonds} holds
 for $t=T$ by Lemma~\ref{lem:unif_conv}.

 Now assume that for some $t\in \ce{0,T{-}1}$, for every accumulation point
 $x_{t+1}^*$ of $\np{x_{t+1}^k}_{k\in \N}$ we have
 \begin{equation}
  \label{HR_sto}
  \overline{V}_{t+1}^*\np{x_{t+1}^*} = \underline{V}_{t+1}^*\np{x_{t+1}^*}.
 \end{equation}
 % \todo[inline]{MA: j'ai change l'ordre car $x^k$ est la trajectoire optimale associee a $ \underline{V}_{t+1}^{k}$ et non $ \underline{V}_{t+1}^{k+1}$ }

 On the one hand, for every index $k\in \N$ one has
 \begin{align}
  \underline{V}_t^{k+1}\np{x_t^k}
   & = \aB{\underline{V}_{t+1}^{k+1}}{x_t^k},  \tag{\text{Tightness}}
  \\
   & \geq \aB{\underline{V}_{t+1}^{k}}{x_t^k},    \tag{\text{Monotonicity}} \\
   & = \Besp{ \pB[\mathbf{W_{t+1}}]{\underline{V}_{t+1}^{k}}{x_t^k}}
  \tag{\text{by definition of $\aB{}{}$}}  %
  \\
   & = \Besp{ \cost[\mathbf{W_{t+1}}]{x_t^k, u_t^{\mathbf{W_{t+1}}}} +
  \underline{V}_{t+1}^{k}\bp{\dyn[\mathbf{W_{t+1}}]{x_t^k, u_t^{\mathbf{W_{t+1}}}}}}
  \tag{by \Cref{eq:pc_control}}
  \\
  %%  & \geq \Besp{ \cost[\mathbf{W_{t+1}}]{x_t^k, u_t^{\mathbf{W_{t+1}}}} +
  %% \underline{V}_{t+1}^{k}\bp{\dyn[\mathbf{W_{t+1}}]{x_t^k, u_t^{\mathbf{W_{t+1}}}}}}
  %% \tag{\text{Monotonicity}}
  %% \\
   & = \sum_{w\in \supp{\mathbf{W_{t+1}}}}
  \Pro\bc{\mathbf{W_{t+1}} = w}
  \Bp{ \cost{x_t^k, u_t^w} + \underline{V}_{t+1}^{k}\bp{\dyn{x_t^k,u_t^w}}} \notag
  \eqfinp
 \end{align}
 On the other hand, for every index $k\in \N$ one has
 \begin{align}
  \overline{V}_t^{k+1}\np{x_t^k}
   & = \aB{\overline{V}_{t+1}^{k+1}}{x_t^k}, \tag{\text{Tightness}}
  \\
   & = \Besp{ \pB[\mathbf{W_{t+1}}]{\overline{V}_{t+1}^{k+1}}{x_t^k}} \notag
  \\
   & \leq \Besp{ \cost[\mathbf{W_{t+1}}]{x_t^k, u_t^{\mathbf{W_{t+1}}}} +
  \overline{V}_{t+1}^{k+1}\bp{\dyn[\mathbf{W_{t+1}}]{x_t^k, u_t^{\mathbf{W_{t+1}}}}}}
  \tag{\text{Def. of pointwise $\pB{}{}$}}
  \\
   & \leq \Besp{ \cost[\mathbf{W_{t+1}}]{x_t^k, u_t^{\mathbf{W_{t+1}}}} +
  \overline{V}_{t+1}^{k}\bp{\dyn[\mathbf{W_{t+1}}]{x_t^k, u_t^{\mathbf{W_{t+1}}}}}}
  \tag{\text{Monotonicity}}
  \\
   & = \sum_{w\in \supp{\mathbf{W_{t+1}}}} \Pro\nc{\mathbf{W_{t+1}} = w}
  \Bp{\cost{x_t^k, u_t^{w}} + \overline{V}_{t+1}^{k}\bp{\dyn{x_t^k,u_t^{w}}}} \notag
  \eqfinp
 \end{align}
 By definition of the problem-child trajectory, recall that $u_t^k := u_t^{w_t^k}$,
 thus we have $x_{t+1}^k := \dyn[w_t^k]{x_t^k, u_t^k}$ and for every $k\in \N$
 \begin{align*}
  0 \leq \overline{V}_t^{k+1}\np{x_t^k} - \underline{V}_t^{k+1}\np{x_t^k}
   & \leq \sum_{w\in \supp{\mathbf{W_{t+1}}}} \bprob{\mathbf{W_{t+1}} = w}
  \Bp{ \np{\overline{V}_{t+1}^{k} - \underline{V}_{t+1}^k}\bp{\dyn{x_t^k,u_t^w}}}
  \\
   & \leq \overline{V}_{t+1}^{k}\np{x_{t+1}^k} - \underline{V}_{t+1}^{k}\np{x_{t+1}^k}
  \eqfinp
 \end{align*}
 Thus, we get that for every function $\sigma : \N \to \N$
 \begin{equation}
  \label{Presque_sto}
  0 \leq \overline{V}_t^{\sigma(k)+1}\np{x_t^{\sigma(k)}} - \underline{V}_t^{\sigma(k)+1}\np{x_t^{\sigma(k)}}
  \leq \overline{V}_{t+1}^{\sigma(k)}\np{x_{t+1}^{\sigma(k)}} - \underline{V}_{t+1}^{\sigma(k)}\np{x_{t+1}^{\sigma(k)}}
  \eqfinp
 \end{equation}
 By Lemma~\ref{lem:conv_commune} and continuity of the dynamics, there exists an
 increasing function $\sigma : \N \to \N$ such that the sequence of future
 states
 $x_{t+1}^{\sigma(k)} =
  \dyn[w_{t+1}^{\sigma(k)}]{x_t^{\sigma(k)},u_t^{\sigma(k)}}$, $k\in \N$,
 converges to some future state $x_{t+1}^* \in X_{t+1}$.  Thus, by
 Lemma~\ref{lem:unif_conv} applied to the $2L_{V_{t+1}}$-Lipschitz functions
 $g^k := \overline{V}_{t+1}^{\sigma(k)} - \underline{V}_{t+1}^{\sigma(k)}$,
 $k\in \N$ and the sequence $y^k := x_{t+1}^{\sigma(k)}$, $k\in \N$ we have that
 \begin{equation*}
  \overline{V}_{t+1}^{\sigma(k)}\np{x_{t+1}^{\sigma(k)}} - \underline{V}_{t+1}^{\sigma(k)}\np{x_{t+1}^{\sigma(k)}}
  \underset{k\rightarrow + \infty}{\longrightarrow}
  \overline{V}_{t+1}^{*}\np{x_{t+1}^{*}} - \underline{V}_{t+1}^{*}\np{x_{t+1}^{*}}
  \eqfinp
 \end{equation*}
 Likewise, by Lemma~\ref{lem:unif_conv} applied to the $2L_{V_t}$-Lipschitz
 functions
 $g^k := \overline{V}_{t}^{\sigma(k)+1} - \underline{V}_{t}^{\sigma(k)+1}$,
 $k\in \N$ and the sequence $y^k := x_{t}^{\sigma(k)}$, $k\in \N$ we have that
 \begin{equation*}
  \overline{V}_{t}^{\sigma(k)+1}\np{x_{t}^{\sigma(k)}} - \underline{V}_{t}^{\sigma(k)+1}\np{x_{t}^{\sigma(k)}}
  \underset{k\rightarrow + \infty}{\longrightarrow}
  \overline{V}_{t}^{*}\np{x_{t}^{*}} - \underline{V}_{t}^{*}\np{x_{t}^{*}}
  \eqfinp
 \end{equation*}
 Thus, taking the limit in  $k$ in \Cref{Presque_sto}, we have that
 \begin{equation*}
  \label{IncreasingGap*_sto}
  0 \leq \overline{V}_t^*\np{x_t^*} - \underline{V}_t^*\np{x_t^*} \leq  \overline{V}_{t+1}^{*}\np{x_{t+1}^{*}} - \underline{V}_{t+1}^{*}\np{x_{t+1}^{*}}.
 \end{equation*}
 By induction hypothesis \eqref{HR_sto} we have that
 $\overline{V}_{t+1}^{*}\np{x_{t+1}^{*}} - \underline{V}_{t+1}^{*}\np{x_{t+1}^{*}} = 0$. Thus,
 we have shown that
 \[
  \overline{V}_t^*\np{x_t^*} = \underline{V}_t^*\np{x_t^*}
  \eqfinp
 \]
 This concludes the proof.
\end{proof}

\section{Illustrations in the linear-polyhedral framework}
\label{sec:numerique}

In this section, we first present a class of Lipschitz MSP that we call
\emph{linear-polyhedral} MSP where dynamics are linear and costs are polyhedral,
\emph{i.e.} functions with convex polyhedral epigraph. Second, we give three
selection functions, one which generates polyhedral lower approximations (see
\S\ref{sec:SDDP}) and two which generates upper approximations, one as infima of
$U$-shaped functions (see \S\ref{sec:U}) and one as infima of $V$-shaped
functions (see \S\ref{sec:V}).

In Table~\ref{recapitulatif_numerique} we illustrate the flexibility made
available by TDP to the decision maker to approximate value functions.
Implementations were done in the programming language Julia 1.4.2 using
the optimization interface JuMP 0.21.3, \cite{Du.Hu.Lu2017}. The code
is available online (\url{https://github.com/BenoitTran/TDP})
as a collection of Julia Notebooks.

\begin{table}[h]
 \centering
 \begin{tabular}{|c||c|c|c|c|}
  \hline
  Selection mapping & Tight  & Valid  & Averaged & Computational difficulty                                           \\
  \hline
  SDDP              & \cmark & \cmark & \cmark   & $\mathrm{Card}\np{\mathbf{W_{t+1}}}$ LPs                           \\
  \hline
  U                 & \cmark & \xmark & \cmark   & $\mathrm{Card}\np{\mathbf{W_{t+1}}} \cdot \mathrm{Card}\np{F}$ QPs \\
  \hline
  V                 & \cmark & \cmark & \xmark   & one LP                                                             \\
  \hline
 \end{tabular}
 \caption{\label{recapitulatif_numerique} Summary of the three selection functions
  presented in Section~\ref{sec:numerique}.}
\end{table}

\subsection{Linear-polyhedral MSP}
We want to solve MSPs where the dynamics are linear and the costs are
polyhedral. That is, we want to solve optimization problems of the form
\eqref{MSP} where for each time step $t \in \ce{0,T{-}1}$ the state dynamics
is linear, $\dyn{x,u} = A^w_t x + B^w_t u$ for some matrices $A_t^w$ and $B_t^w$ of
coherent dimensions and the cost is polyhedral:
\begin{equation}
 \label{eq:polyhedralcosts}
 \cost{x,u} = \max_{i \in I_t} \bscal{c^{i,w}_t}{(x;u)} + d_t^{i,w} + \delta_{P_t^w}(x,u)
 \eqfinp
\end{equation}
where $I_t$ is a finite set, $c^{i,w} \in \X{\times}\U$, $d_t^{i,w}$ is a scalar
and $P_t^w$ is a convex polyhedron. The final cost function $\psi$ is of the
form $\psi(x) = \max_{i \in I_T} \nscal{c^{i}_T}{x} + d_T^{i}+ \delta_{X_{T}}$
where ${X_{T}}$ is a nonempty convex polytope. We assume that
Assumption~\ref{whitenoise}, \ref{compactStates} and \ref{recourse} are
satisfied.

\begin{proposition}[Linear-polyhedral MSP are Lipschitz MSP]
 \label{linpol_are_lipschitz}
 Linear-polyhedral MSP are Lipschitz MSP in the sense of Assumption~\ref{LipschitzMSP}.
\end{proposition}

\begin{proof}
 By construction, the costs $\cost{}{}$ and the dynamics $\dyn{}{}$ are
 Lipschitz continuous with explicit constants. We show that for every
 $t\in \ce{0,T{-}1}$ and each $w\in \supp{\mathbf{W_{t+1}}}$, the constraint
 set-valued mapping $\cU{}{}$ is Lipschitz continuous. From \cite[Example
  9.35]{Ro.We2009}, it is enough to show that the graph of $\cU{}{}$ is a convex
 polyhedron. By assumption $\dom\, \cost{}{}$ is a convex polyhedron and by
 recourse $\mgraph\, \cU{}{}$ is nonempty. As a nonempty intersection of convex
 polyhedron is a convex polyhedron, we only have to show that
 $\nset{\np{x,u} \in \X{\times}\U}{\dyn{x,u} \in X_{t+1}}$ is a convex
 polyhedron as well.

 Using Equation~\eqref{Xt} we have that $X_{t+1}$ is given by
 $X_{t+1} = \cap_{w \in \supp{\mathbf{W_{t+2}}}}\projname_{\X} \bp{ \dom
   \,\costname_{t+1}^w}$, which is the nonempty intersection of convex
 polyhedron. Thus, $X_{t+1}$ is a convex polyhedron which implies that
 there exist a matrix $Q_{t+1}$ and a vector $b_{t+1}$ such that
 $X_{t+1} = \bset{{x}\in \X}{ Q_{t+1} x \leq b_{t+1}}$.
 Therefore, we obtain that the two following sets coincide
 $$
  \bset{ \np{x,u}\in \X{\times}\U}{\dyn{x,u} \in X_{t+1}}
  =
  \bset{ \np{x,u}\in \X{\times}\U}{Q_{t+1}A_t^wx +  Q_{t+1}B_t^wu \le b_{t+1}}
  \eqfinp
 $$
 The latter being convex polyhedral we obtain that the former is convex polyhedral.
 This ends the proof.
\end{proof}

Now, observe that as linear-polyhedral MSP are Lipschitz MSP, by
Corollary~\ref{Lipschitz_Vt}, the value function $V_t$ is $L_{V_t}$-Lipschitz
continuous on $X_t$ for all $t\in \ic{0,T}$. Moreover, under the recourse assumption~\ref{recourse} we
can show that the Bellman operators ${\aB{}{}}_{t\in \ic{0,T{-}1}}$ preserves polyhedrality in the
sense defined below.

\begin{lemma}[$\aB{}{}$ preserves polyhedrality]
 \label{polypreserved}
 For every $t\in \ce{0,T{-}1}$, if $\func : \X \to \Rb$ is a \emph{polyhedral
  function}, \emph{i.e.} its epigraph is a convex polyhedron, then
 $\aB{\func}{}$ is a polyhedral function as well.
\end{lemma}
\begin{proof}
 For every $w\in \supp{\mathbf{W_{t+1}}}$, we have shown in the proof of
 Proposition~\ref{linpol_are_lipschitz} that the graph of $\cU{}{}$ is a convex
 polyhedron. Thus,
 $\np{x,u} \mapsto \cost{x,u} + \func\bp{\dyn{x,u}} + \delta_{\mgraph
   \cU{}{}}\np{x,u}$ is convex polyhedral and by \cite[Proposition
  5.1.8.e]{Bo.Le2006}, $\pB{\func}{}$ is polyhedral as well.  Finally, under
 Assumption~\ref{whitenoise}, we deduce that
 $\aB{\func}{} := \sum_{w \in \supp{\mathbf{W_{t+1}}}} \pB{\func}{}$ is
 polyhedral as a finite sum of polyhedral functions. This ends the proof.
\end{proof}

\subsection{SDDP lower approximations}
\label{sec:SDDP}

Stochastic Dual Dynamic Programming is a popular algorithm which was introduced
by Perreira and Pinto in 1991 (see \cite{Pe.Pi1991}) and studied extensively
since then, \emph{e.g.} \cite{Ah.Ca.da2019,Ba.Do.Za2018,Gu2014,Ph.Gu2008,Ph.de.Fi2013,Sh2011,Zo.Ah.Su2018}.

Lemma~\ref{polypreserved} is the main intuitive justification of using SDDP in
linear-polyhedral MSPs: if the final cost function is polyhedral, as the operators
$\na{\aB{}{}}_{t\in \ic{0,T{-}1}}$ preserve polyhedrality, by backward induction on time, we obtain
that the value function $V_t$ is polyhedral for every $t\in \ce{0,T}$. Hence, the
decision maker might be tempted to construct polyhedral approximations of $V_t$
as well.

We now present a way to generate polyhedral lower approximations of
value functions, as done in the literature of SDDP, by defining a proper selection
mapping. When the value functions are convex, it builds lower approximations as suprema
of affine cuts. We put SDDP in TDP's framework by constructing a lower selection function.

First, for every time step $t\in \ce{0,T}$, define the set of basic functions,
\[
 \lFuncb_t^{\mathrm{SDDP}} :=
 \bset{\nscal{a}{\cdot} + b + \delta_{X_t}}{\np{a,b} \in \X{\times}\R \ \text{s.t.} \ \norm{a} \leq L_{V_t}}
 \eqfinp
\]
At time $t = T$, given a trial point $x \in X_T$, we define
$\underline{S}_{T}^{\text{SDDP}}\np{ x} = \nscal{a_x}{\cdot -x} + b_x$, where $a_x$
is a subgradient of the convex polyhedral function $\psi$ at $x$ and
$b_x = \psi\np{x}$. Tightness and validity of $\underline{S}_{T}^{\text{SDDP}}$
follows from the given expression. Now, for $t \in \ic{0,T-1}$, we compute a
tight and valid cut for $\pB{}{}$ for each possible value of the noise $w$ then
average it to get a tight and valid cut for $\aB{}{}$. The details are given in
Algorithm~\ref{SDDP_Selection}.

\begin{algorithm}
 \caption{\label{SDDP_Selection}SDDP Selection function $\underline{S}_{t}^{\mathrm{SDDP}}$ for $t<T$}
 \begin{algorithmic}
  \REQUIRE{A set of basic functions $\lFunc_{t+1} \subset \lFuncb_{t+1}^{\mathrm{SDDP}}$ and a trial point $x_t \in X_t$.}
  \ENSURE{A tight and valid basic function $\lfunc_t \in \lFuncb_{t}^{\mathrm{SDDP}}$.}
  \FOR{$w\in \supp{\mathbf{W_{t+1}}}$}
  \STATE{Solve by linear programming $b^w := \pB{\lvopt_{\lFunc_{t+1}}}{x}$ and compute a subgradient $a^w$ of $\pB{\lvopt_{\lFunc_{t+1}}}{}$ at $x$.}
  \ENDFOR
  \STATE{Set $\lfunc := \langle a, \cdot \rangle + b  + \delta_{X_t}$ where $a := \sum_{w\in \supp{\mathbf{W_{t+1}} }} p_w a^w$ and $b = \sum_{w\in \supp{\mathbf{W_{t+1}}}} p_w b^w$.}
 \end{algorithmic}
\end{algorithm}

We say that $\ZSelection[w]{}$ is a selection function for $\pB{}{}$, for a given
noise value $w \in \supp{\mathbf{W_{t+1}}}$ if
Definition~\ref{CompatibleSelection} is satisfied when replacing $\aB{}{}$ by
$\pB{}{}$. We now prove that $\underline{S}_{t}^{\text{SDDP}}$ is a selection
function, \emph{i.e.} it is tight and valid in the sense of
Definition~\ref{CompatibleSelection}. It follows from the general fact that by
averaging functions which are tight and valid for the pointwise Bellman
operators $\pB{}{}$, $w \in \supp{\mathbf{W_{t+1}}}$, then one get a tight and
valid function for the average Bellman operator $\aB{}{}$. Note that the average
of affine functions is still an affine function, the set of basic functions
$\lFuncb_t^{\mathrm{SDDP}}$ is stable by averaging.

\begin{lemma}
 \label{pointwisetoaverage}
 Let a time step $t \in \ce{0,T{-}1}$ be fixed and let be given for every noise
 value $w\in \supp{\mathbf{W_{t+1}}}$ a selection function
 $\ZSelection[w]{}$ for $\pB{}{}$. Then, the mapping $\ZSelection[]{}$ defined
 by $\ZSelection[]{} = \E\nc{\ZSelection[\mathbf{W_{t+1}}]{}}$ is a selection
 mapping for $\aB{}{}$.
\end{lemma}

\begin{proof}
 Fix $t\in \ce{0,T{-}1}$. Given a trial point $x\in X_t$ and a set of basic
 functions $\Func$, the pointwise tightness (resp. validity) equality
 (resp. inequality) is satisfied for every realization $w$ of the noise
 $\mathbf{W_{t+1}}$, that is
 \begin{align}
   & \ZSelection[w]{F, x}\np{x} = \pB{\vopt_{\Func}}{x}, \tag{Pointwise tightness}                                 \\
   & \ZSelection[w]{F, x} \geq \pB{\uvopt_{\Func}}{} , \tag{Pointwise validity when building upper approximations} \\
   & \ZSelection[w]{F, x} \leq \pB{\lvopt_{\Func}}{} \tag{Pointwise validity when building lower approximations}.
 \end{align}
 Recall that
 $\aB{\vopt_{\Func}}{x} = \E \nc{\pB[\mathbf{W_{t+1}}]{\vopt_{\Func}}{x}}$, thus
 taking the expectation in the above equality and inequalities, one gets the
 lemma.
\end{proof}

\begin{proposition}[SDDP Selection function]
 For every $t\in \ce{0,T}$, the mapping $\underline{S}_{t}^{\mathrm{SDDP}}$ is
 a selection function in the sense of Definition~\ref{CompatibleSelection}.
\end{proposition}

\begin{proof}
 For $t=T$, for every $x_T\in X_T$, by construction we have
 \[
  \underline{S}_{T}^{\mathrm{SDDP}}\np{x_T} = \psi\np{x_T} = V_T\np{x_T}.
 \]
 Thus, $\underline{S}_{T}^{\mathrm{SDDP}}$ is tight and it is valid as
 $\underline{S}_{T}^{\mathrm{SDDP}}\np{x_T} = \langle a, \cdot - x_T \rangle +
  \psi\np{x_T}$ is an affine minorant of the convex function $\psi$ which is
 exact at $x_T$.  Now, fix $t \in\ic{0,T{-}1}$, a set of basic functions
 $\lFunc_t \subset \lFuncb_{t}^{\mathrm{SDDP}}$ and a trial point $x_t \in
  X_t$. By construction, $\underline{S}_{t}^{\mathrm{SDDP}}$ is tight as we have
 \[
  \underline{S}_{t}^{\mathrm{SDDP}}\bp{\lFunc_t, x_t}\np{x_t} =
  \nscal{a}{x_t - x_t} + \Besp{ \pB[\mathbf{W_{t+1}}]{\lvopt_{\lFunc_t}}{x_t}} = \aB{\lvopt_{\lFunc_t}}{x_t}.
 \]
 Moreover, for every $w\in \supp{\mathbf{W_{t+1}}}$, $a^w$ (see
 Algorithm~\ref{SDDP_Selection}) is a subgradient of $\pB{\lvopt_{\lFunc_t}}{}$
 at $x_t$. Thus as $a$ is equal to $\nesp{a^{\mathbf{W_{t+1}}}}$ it is a
 subgradient of $\aB{\lvopt_{\lFunc_t}}{}$ at $x_t$. Hence, the mapping
 $\underline{S}_{t}^{\mathrm{SDDP}}$ is valid.
\end{proof}

\subsection{$U$-upper approximations}
\label{sec:U}
We have seen in Lemma~\ref{pointwisetoaverage}, that in order to construct a
selection function for $\aB{}{}$, it suffices to construct a selection function
for each pointwise Bellman operator $\pB{}{}$. In order to do so, for upper
approximations we exploit the min-additivity of the pointwise Bellman operators
$\pB{}{}$. That is, given a set of functions $\Func$, we use the following
decomposition
\[
 \forall t\in \ce{0,T{-}1}, \forall x\in \X, \forall w \in \supp{\mathbf{W_{t+1}}}, \ \pB{\inf_{\func \in \Func} \func}{x} = \inf_{\func\in \Func}   \pB{\func}{x}.
\]
This is a decomposition of the computation of $\pB{\uvopt_{\Func}}{}$ which is
possible for upper approximations but not for lower approximations as for
minimization problems, the Bellman operators (average or pointwise) are min-plus
linear but generally not max-plus linear.

However, in linear-polyhedral MSP, the value functions are
polyhedral. Approximating from above value function $V_t$ by infima of convex
quadratics is not suited: in particular, one cannot ensure validity of a
quadratic at a kink of the polyhedral function $V_t$. Still, we present a
selection function which is tight but not valid. In the numerical experiment of
Figure~\ref{figure:U-SDDP}, we illustrate that the selection function defined below
might not be valid, but the error is still reasonable. Yet, this will motivate
the use of other basic functions more suited to the linear-polyhedral framework,
as done in \S\ref{sec:V}.

We consider basic functions that are $U$-shaped, \emph{i.e.} of the form
$\frac{c}{2}\lVert x - a \rVert^2 + b$ for some constant $c>0$, vector $a$ and
scalar $b$. We call such function a \emph{$c$-function}. We now fix a sequence
of constants $\np{c_t}_{t\in \ce{0,T}}$ such that $c_t > L_{V_t}$.  For every
time $t\in \ce{0,T}$, define the set of basic functions
\[
 \uFuncb_t^{\mathrm{U}} =
 \bset{ \frac{c_t}{2}\lVert x - a \rVert^2 + b + \delta_{X_t} }
 {\np{a,b} \in \X{\times}\R }
 \eqfinp
\]
At time $t=T$, we select the $c_T$-quadratic mapping which is equal to $\psi$ at point
$x\in X_T$ and has same (sub)gradient at $x$, \emph{i.e.}
$\overline{S}_{T}^{\text{U}}\np{x} = \frac{c_T}{2}\lVert \cdot - a \rVert^2 + b$
where $a = x - \frac{1}{c}\lambda$ and
$b = \psi\np{x} - \frac{1}{2c}\lVert \lambda \rVert^2$ with $\lambda$ being a
subgradient of $\psi$ at $x$.

The mapping $\overline{S}_{t}^{\text{U}}$ defined in Algorithm~\ref{U_Selection}
is tight but not necessarily valid, see an illustration in
Figure~\ref{figure:U-SDDP}. As with SDDP, in order to build a tight selection
function at $t<T$ for $\aB{}{}$ we first compute a tight selection function for
each $\pB{}{}$, $w\in \supp{\mathbf{W_{t+1}}}$, which can be done numerically by
quadratic programming.

\begin{algorithm}
 \caption{\label{U_Selection}U Selection function $\overline{S}_{t}^{\text{U}}$ for $t<T$}
 \begin{algorithmic}
  \REQUIRE{A set of basic functions $\uFunc_{t+1} \subset \uFuncb_{t+1}^{\mathrm{U}}$ and a trial point $x_t \in X_t$.}
  \ENSURE{A tight basic function $\ufunc_t \in \uFuncb_{t}^{\mathrm{U}}$.}
  \FOR{$w\in \supp{\mathbf{W_{t+1}}}$}
  \STATE{Solve by quadratic programming $v^w := \pB{\uvopt_{\uFunc_{t+1}}}{x} = \inf_{\ufunc \in \uFunc_{t+1}}\pB{\ufunc}{x}$ and compute $a^w =x - \frac{1}{c}\lambda$ and $b^w = v^w - \frac{1}{2c} \lVert \lambda \rVert^2$ with $\lambda$ being a subgradient of $\pB{\uvopt_{\uFunc_{t+1}}}{}$ at $x$.}
  \ENDFOR
  \STATE{Set $\ufunc := \frac{c_t}{2}\lVert \cdot - a \rVert^2 + b + \delta_{X_t}$  where $a := \E\nc{a^{\mathbf{W_{t+1}}}}$ and $b = \E\nc{ \frac{c_t}{2}\lVert \cdot - a \rVert^2 + b^{\mathbf{W_{t+1}}}}$.}
 \end{algorithmic}
\end{algorithm}

\begin{figure}
 \begin{center}
  \begin{subfigure}[b]{0.30\textwidth}
   \includegraphics[width=\linewidth]{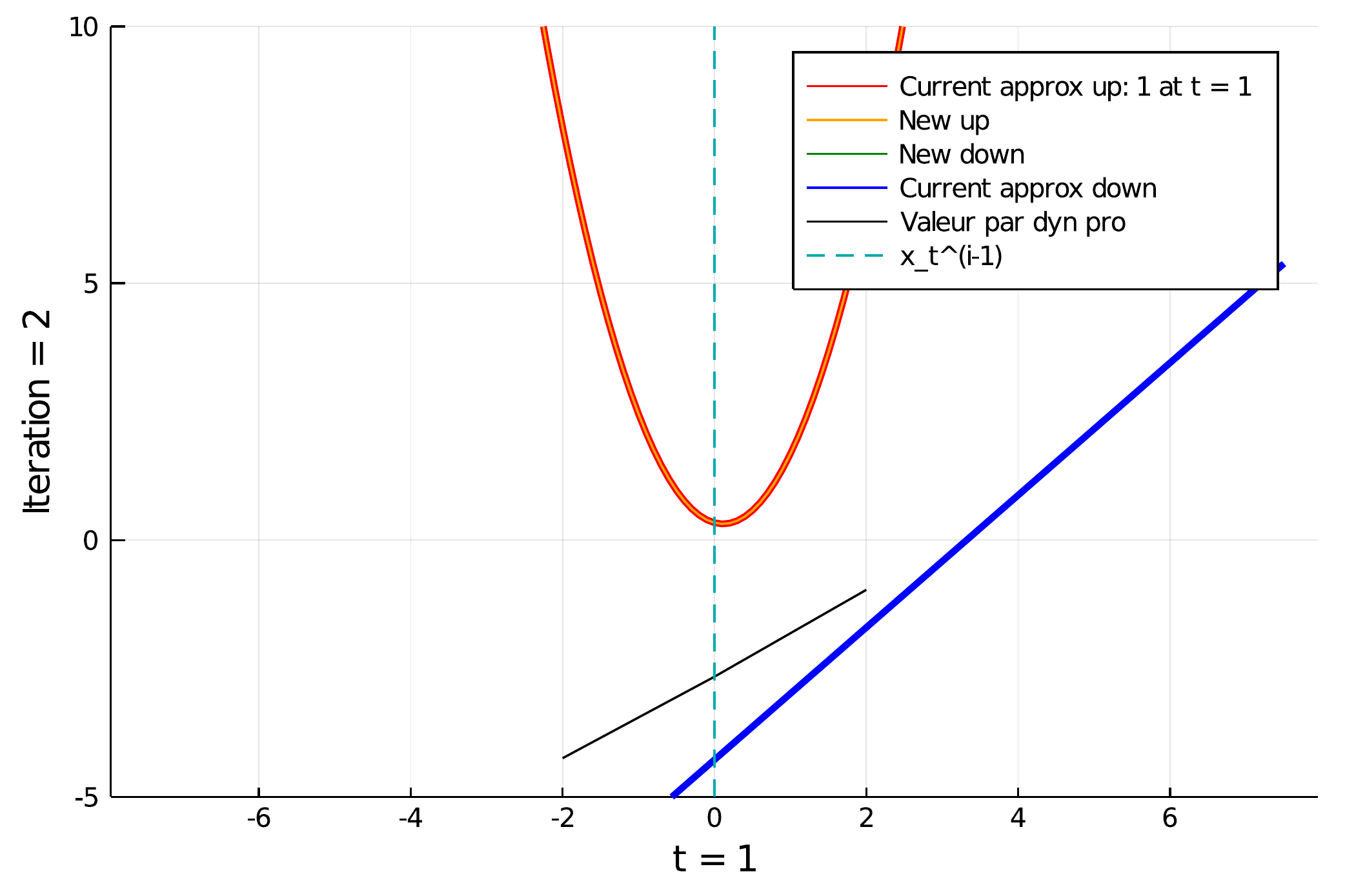}
  \end{subfigure}
  \begin{subfigure}[b]{0.3\textwidth}
   \includegraphics[width=\linewidth]{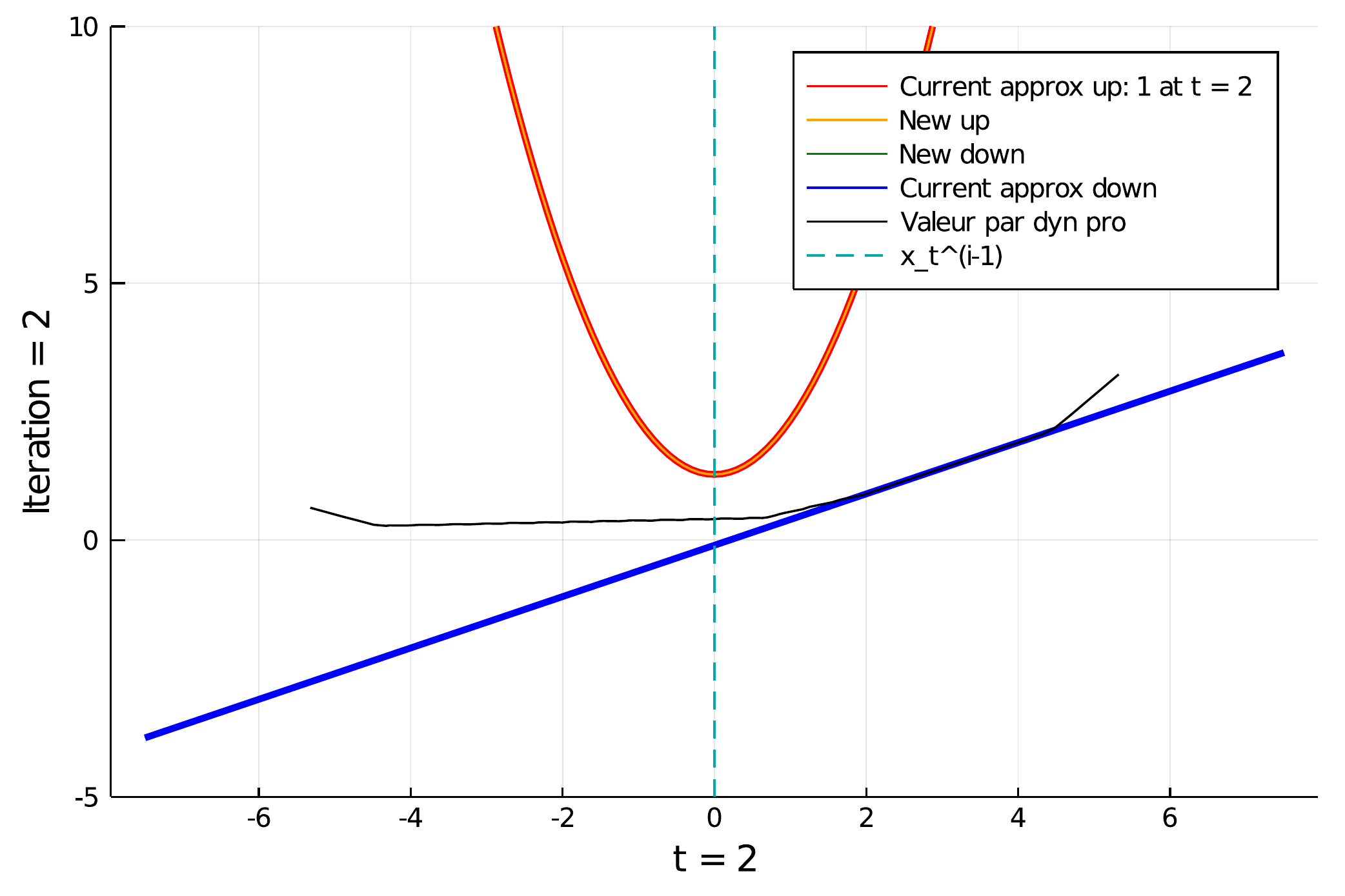}
  \end{subfigure}
  \begin{subfigure}[b]{0.3\textwidth}
   \includegraphics[width=\linewidth]{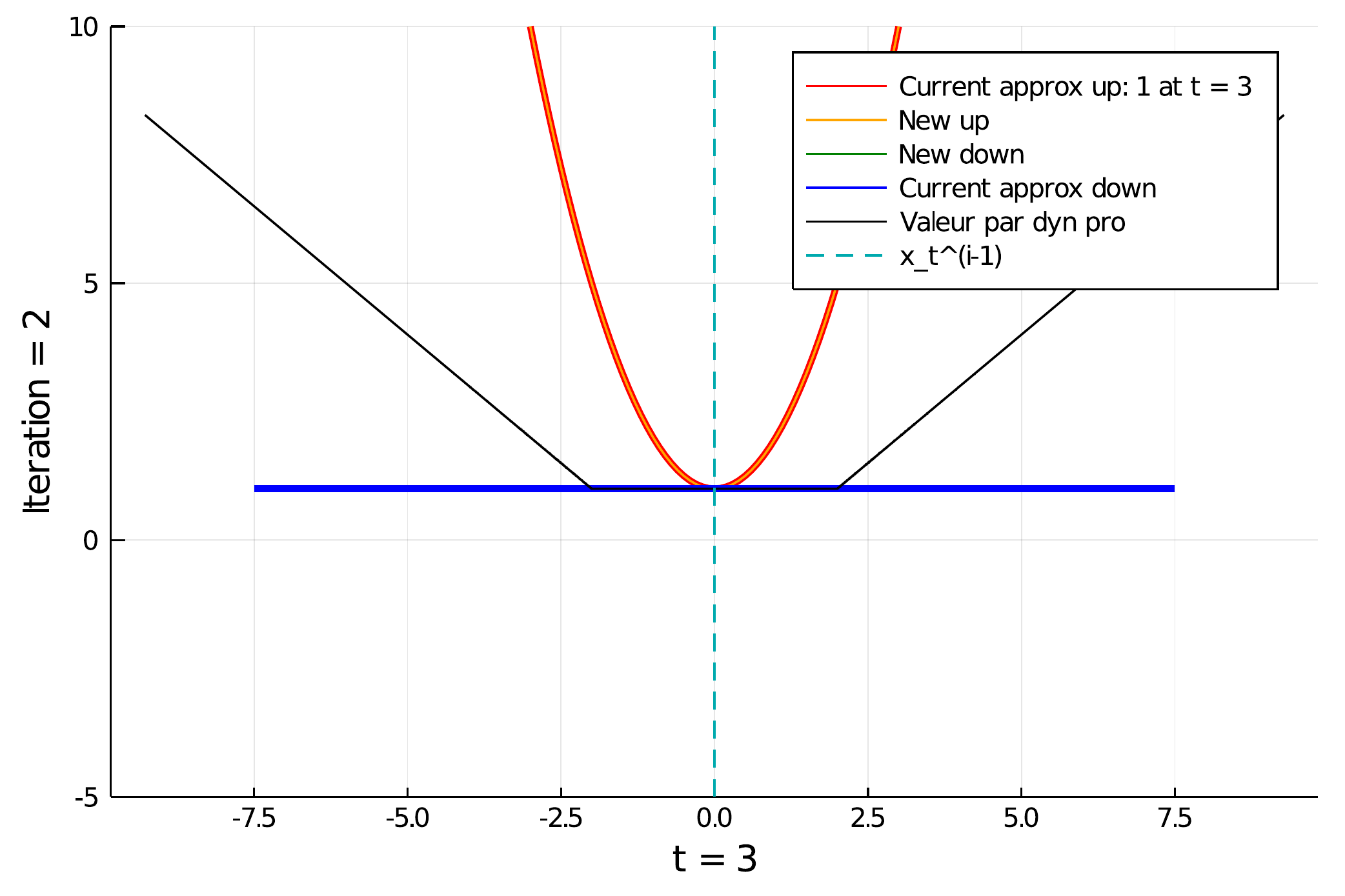}
  \end{subfigure}

  \begin{subfigure}[b]{0.30\textwidth}
   \includegraphics[width=\linewidth]{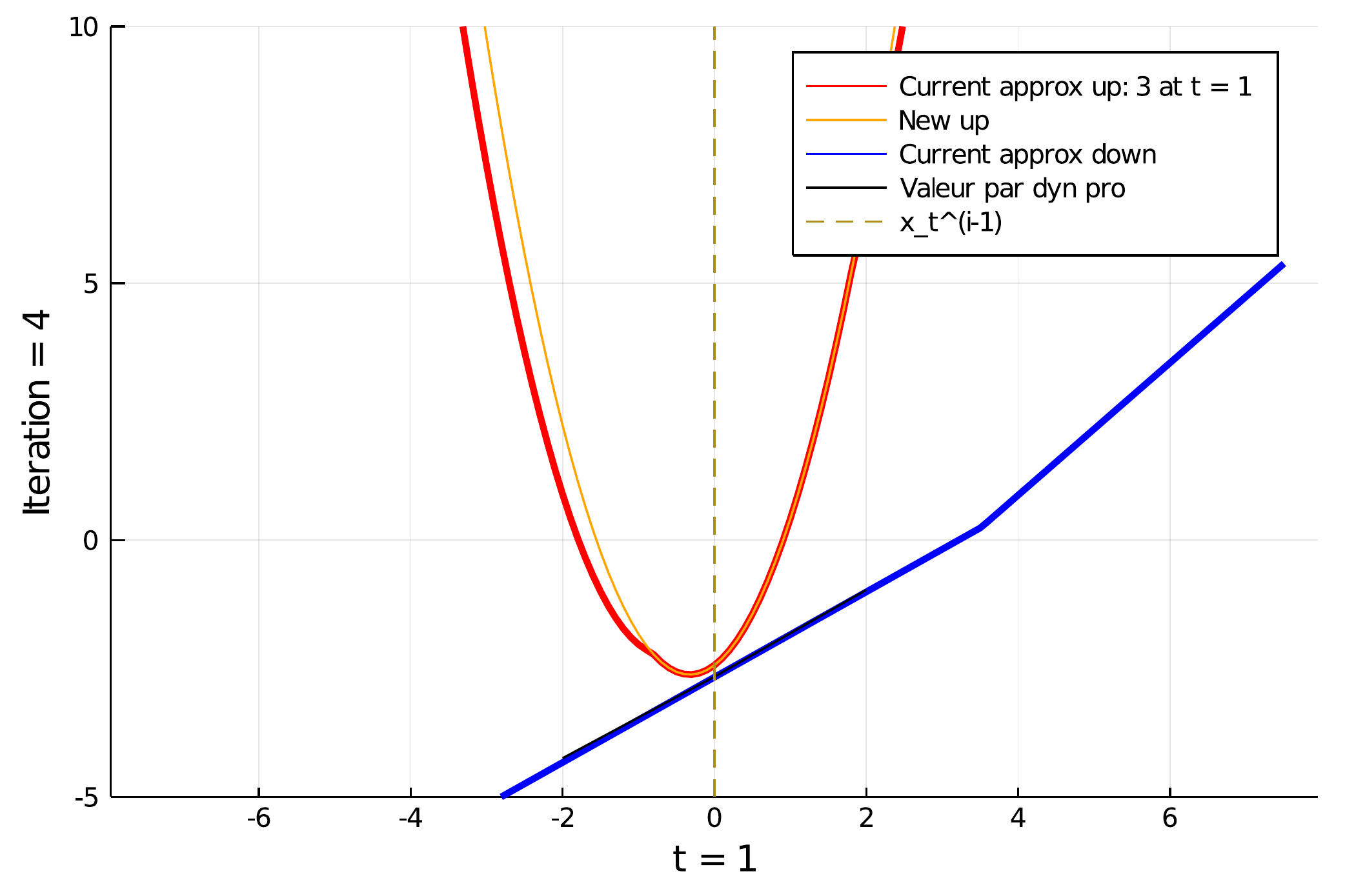}
  \end{subfigure}
  \begin{subfigure}[b]{0.3\textwidth}
   \includegraphics[width=\linewidth]{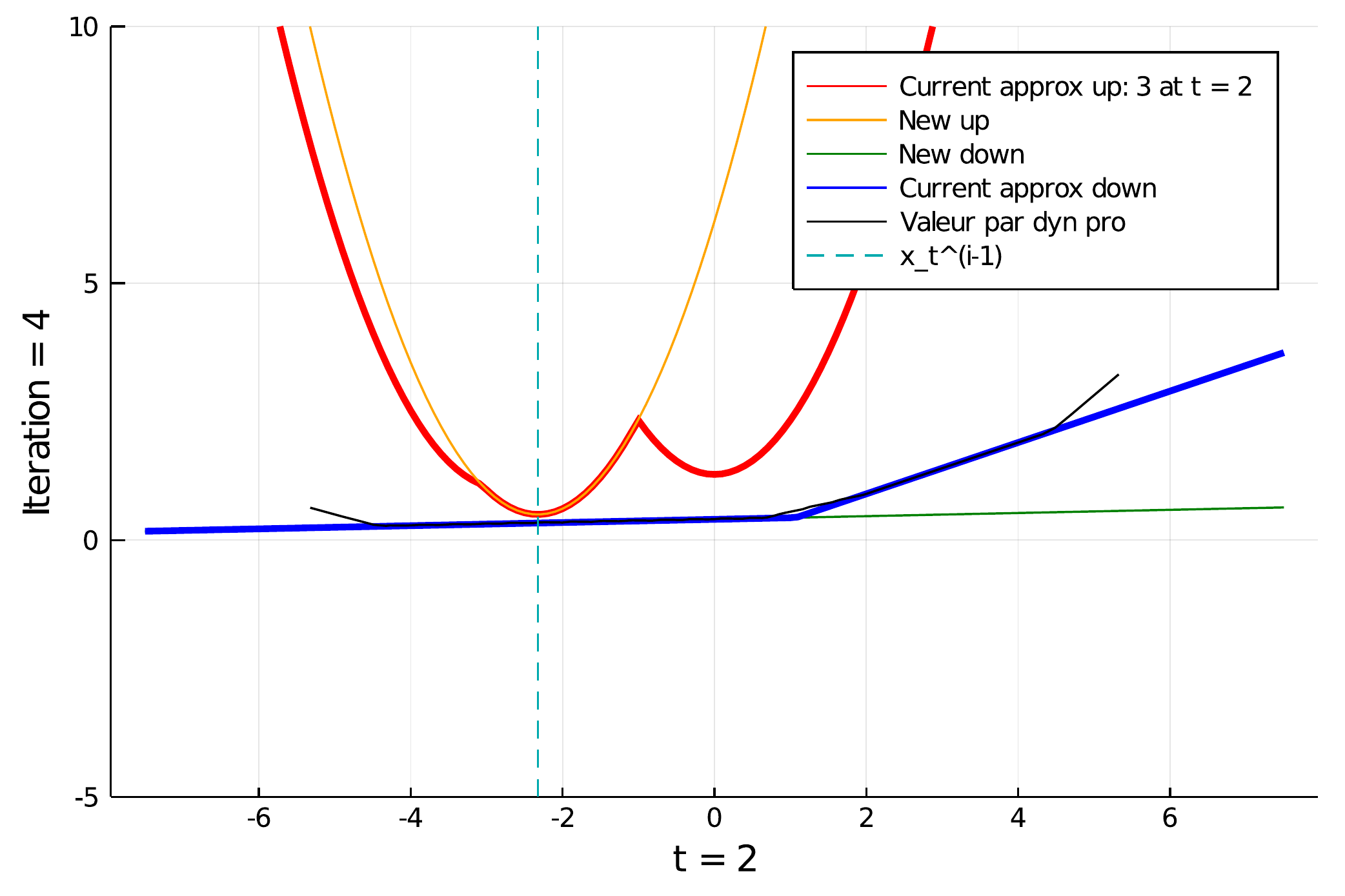}
  \end{subfigure}
  \begin{subfigure}[b]{0.3\textwidth}
   \includegraphics[width=\linewidth]{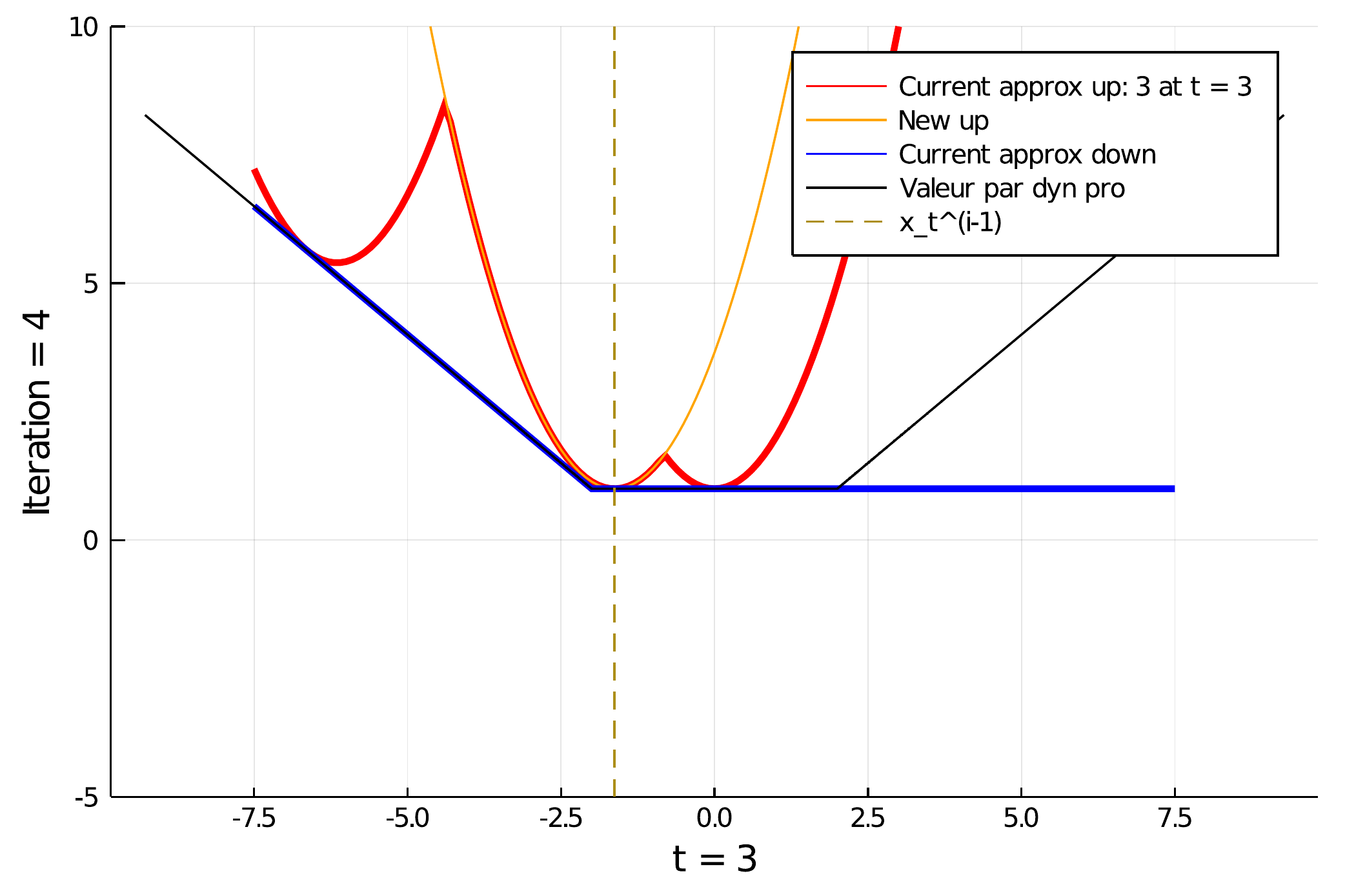}
  \end{subfigure}

  \begin{subfigure}[b]{0.30\textwidth}
   \includegraphics[width=\linewidth]{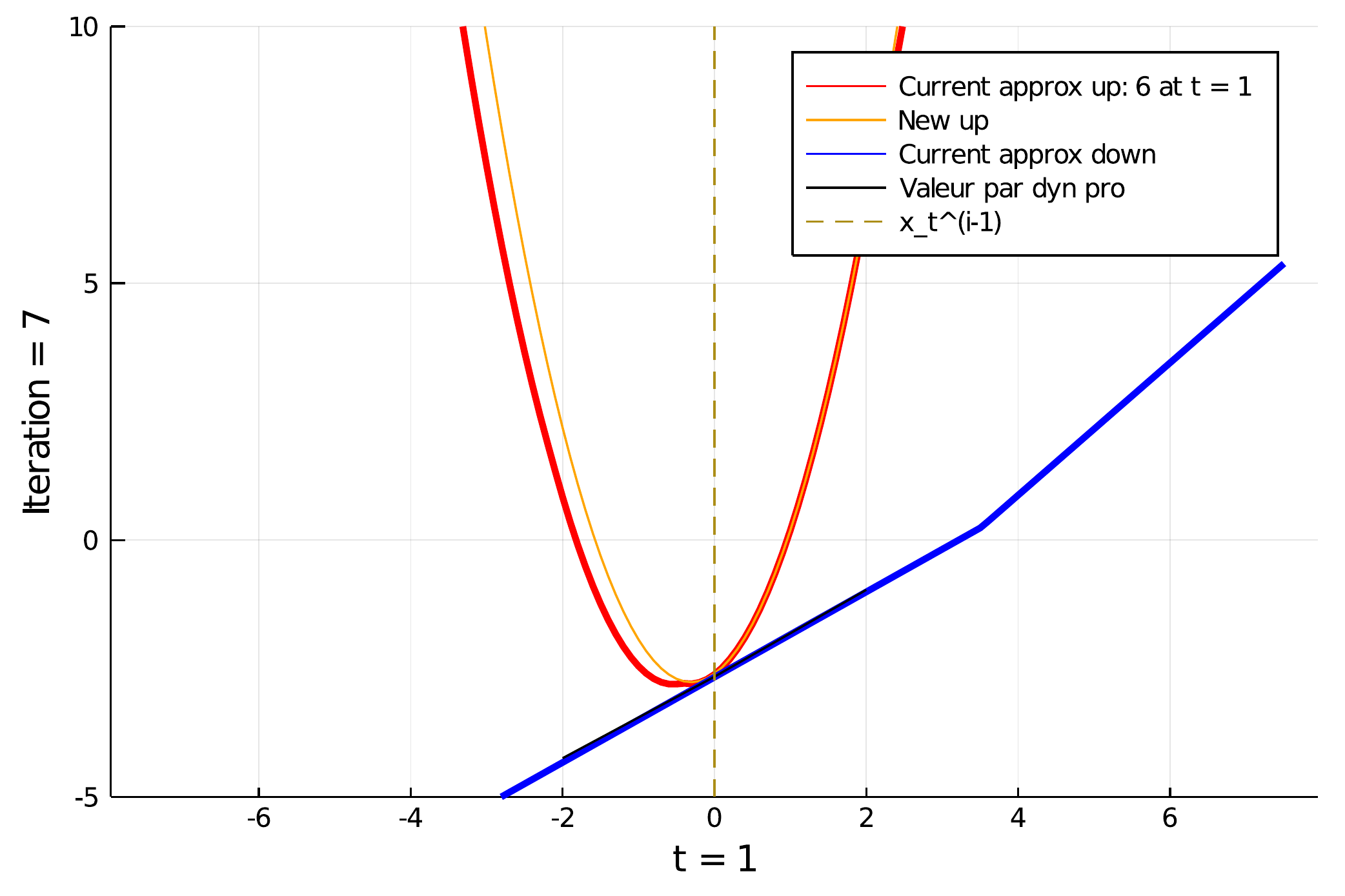}
  \end{subfigure}
  \begin{subfigure}[b]{0.3\textwidth}
   \includegraphics[width=\linewidth]{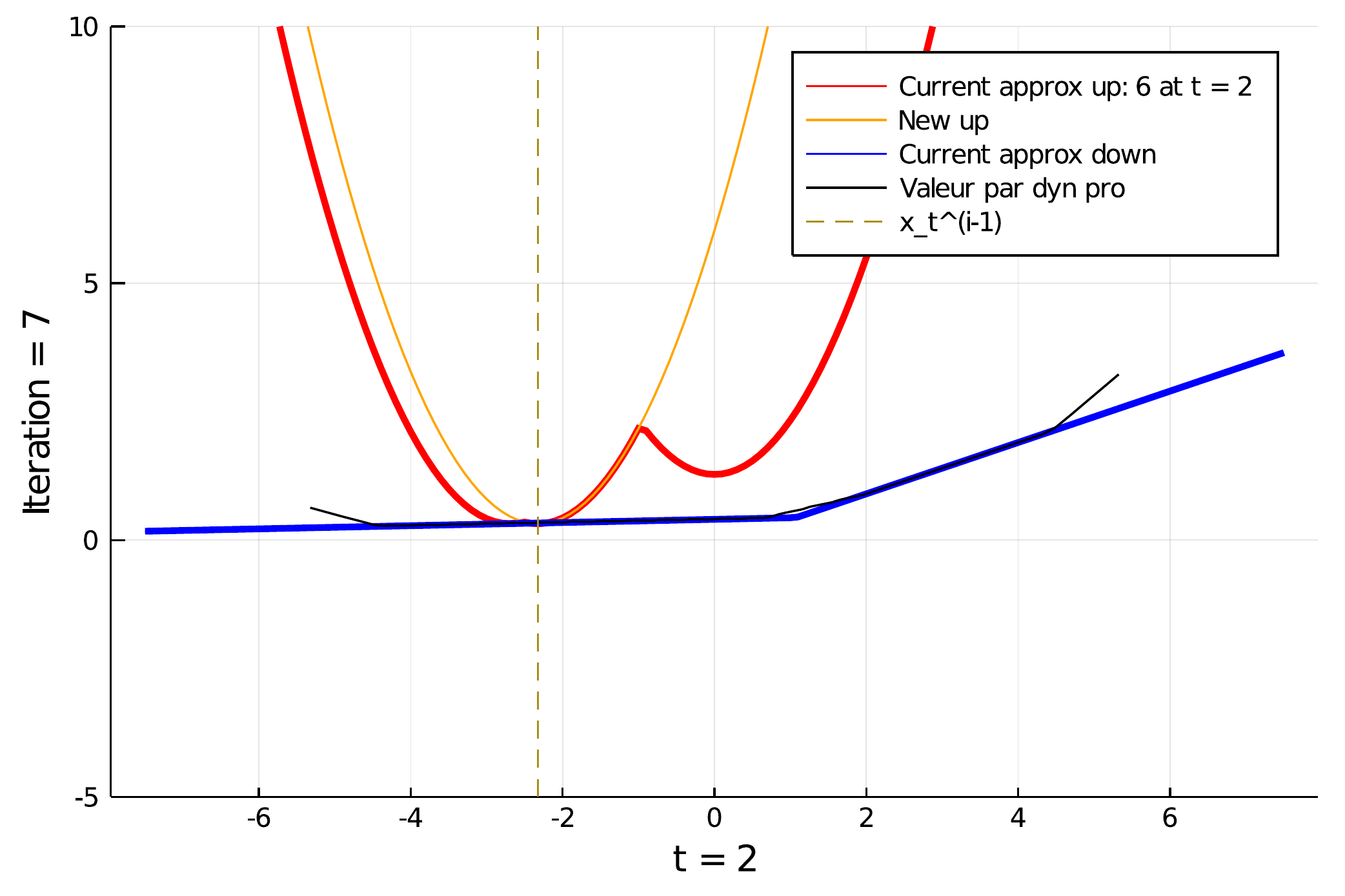}
  \end{subfigure}
  \begin{subfigure}[b]{0.3\textwidth}
   \includegraphics[width=\linewidth]{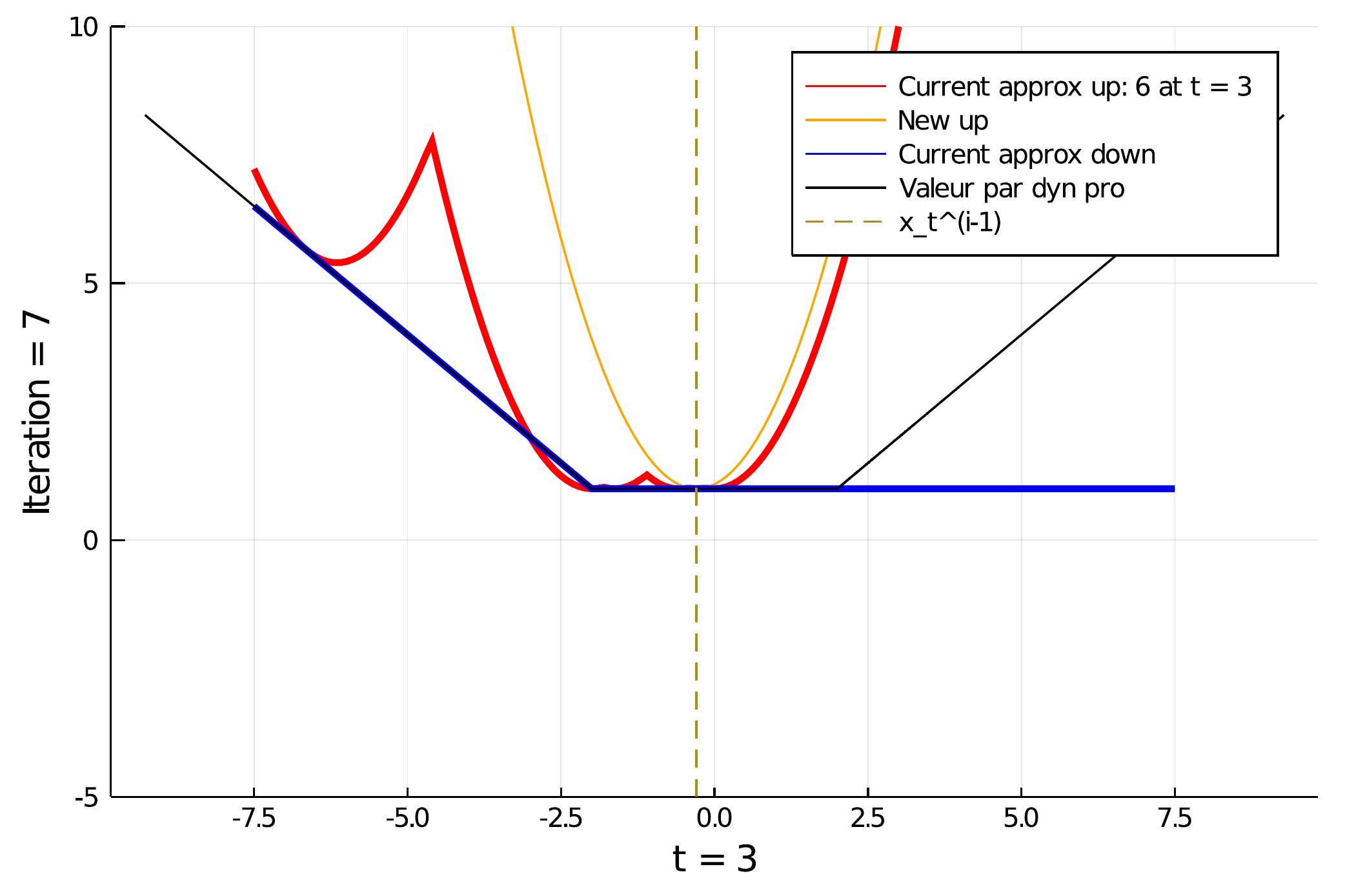}
  \end{subfigure}
 \end{center}
 \caption{\label{figure:U-SDDP} U-SDDP approximations of the value functions. In the bottom right we see that the $U$-shaped basic functions might not be valid when the trial point is associated with a kink of value function. Still, we observe that the gap between upper and lower approximations vanishes along the problem-child trajectory (in dashed lines).}
\end{figure}

\subsection{$V$-upper approximations}
\label{sec:V}
We have seen in \S\ref{sec:U} that $U$-shaped basic functions may not be suited
to approximate polyhedral functions. In \cite{Ph.de.Fi2013}, upper
approximations which were polyhedral as well were introduced. In this section we
propose upper approximations of $V_t$ as infima of $V$-shaped functions. Even
though when $V_t$ is polyhedral the approach of \cite{Ph.de.Fi2013} seems the
most natural, their approximations cannot be easily expressed as a pointwise infima
of basic functions.

In future works we will add a max-plus/min-plus projection
step to TDP in order to broaden the possibilities of converging approximations
available to the decision maker. In particular, polyhedral upper approximations
as in~\cite{Ph.de.Fi2013} will be covered.

In this section, by introducing a new tight and valid selection function, we
would like to emphasize on the flexibility already available to the decision
maker by adopting the framework of TDP.

We consider $V$-shaped functions, \emph{i.e.} functions of the form
$L \lVert x - a \rVert_1 + b$ with $a \in \X = \R^n$ and $b\in \R$ and a
constant $L>0$. We define for every time step $t\in \ce{0,T}$, the set of basic
functions
\[\uFuncb_{t}^{\mathrm{V}} := \Bset{ \frac{L_{V_t}}{\sqrt{n}} \norm{\cdot - a}_1 + b}
 { \np{a,b} \in \X{\times}\R }
 \eqfinp
\]

At time $t = T$, we compute a $V$-shaped function at $\psi\np{x}$, \emph{i.e.}
given a trial point $x \in X_T$, using the expression
$\overline{S}_{T}^{\mathrm{V}}\np{x} = \frac{L_{V_T}}{\sqrt{n}} \lVert \cdot - x
 \rVert_1 + \psi\np{x}$. For time $t \in\ic{0,T{-}1}$, the selection function is
given in Algorithm~\ref{V_Selection}. The main difference with the previous
cases treated in~\S\ref{sec:SDDP} and in~\S\ref{sec:U} is that $V$-shaped function are not
stable by averaging as the average of several $V$-shaped function is a polyhedral function.

\begin{algorithm}
 \caption{\label{V_Selection}V Selection function $\overline{S}_{t}^{\mathrm{V}}$ for $t<T$}
 \begin{algorithmic}
  \REQUIRE{A set of basic functions $\uFunc_{t+1} \subset \uFuncb_{t+1}^{\mathrm{V}}$ and a trial point $x_t \in X_t$.}
  \ENSURE{A tight and valid basic function $\ufunc_t \in \uFuncb_{t}^{\mathrm{V}}$.}
  \STATE{Solve by linear programming $b := \aB{\uvopt_{\uFunc_{t+1}}}{x_t}$.}
  \STATE{Set $\ufunc_t := \frac{L_{V_t}}{\sqrt{n}} \lVert \cdot - x_t \rVert + b$.}
 \end{algorithmic}
\end{algorithm}

\begin{proposition}[V Selection function]
 For every $t\in \ce{0,T}$, the mapping $\overline{S}_{t}^{\mathrm{V}}$
 described in Algorithm~\ref{V_Selection}
 is a selection function in the sense of Definition~\ref{CompatibleSelection}.
\end{proposition}
\begin{proof}
 At time $t=T$, for every $x_T \in X_T$, we have $\overline{S}_{T}^{\mathrm{V}}\np{x_T} = \frac{L_{V_T}}{\sqrt{n}} \lVert \cdot - x_T \rVert_1 + \psi\np{x_T}$. Thus, $\overline{S}_{T}^{\mathrm{V}}\np{x_T}\np{x_T} = \psi\np{x_T}$ and $\overline{S}_{T}^{\mathrm{V}}$ is a tight mapping. As the polyhedral function $\psi(x) = \max_{i \in I_T} \langle c^{i}_T, x \rangle + d_T^i + \delta_{X_{T}}$ is $L_{V_T}$-Lipschitz continuous, by Cauchy-Schwarz inequality, for every $x \in X_T$ and $i\in I_T$, we have
 \[
  \langle c^{i}_T, x - x_T \rangle \leq \lVert c^{i}_T \rVert_2 \lVert x - x_T \rVert_2 \leq  L_{V_T} \frac{1}{\sqrt{n}}  \lVert x-x_T \rVert_1.
 \]
 Adding $\langle c^i_T, x_T \rangle + d_T^i$ on both sides of the last inequality and taking the maximum over $i \in I_T$ we have that
 \[
  \psi\np{x} = \max_{i \in I_T}  \langle c^i_T, x \rangle + d_T^i \leq  L_{V_T} \frac{1}{\sqrt{n}}  \lVert x-x_T \rVert_1 + \psi\np{x_T} = \overline{S}_{T}^{\mathrm{V}}\np{x_T}\np{x},
 \]
 which gives that $\overline{S}_{T}^{\mathrm{V}}$ is a valid mapping.

 Now, fix $t < T$, we show that the mapping $\overline{S}_{t}^{\mathrm{V}}$ is
 tight and valid as well. By construction, for every set of basic functions
 $\uFunc_{t+1} \subset \uFuncb_{t+1}^{\mathrm{U}}$ and trial point
 $x_t \in X_t$, we have
 \[
  \overline{S}_{t}^{\mathrm{V}}\np{\uFunc_{t+1}, x_t}\np{x_t} = b = \aB{\uvopt_{\uFunc_{t+1}}}{x_t}.
 \]
 Hence, $\overline{S}_{t}^{\mathrm{V}}$ is a tight mapping.

 We check that $\overline{S}_{t}^{\mathrm{V}}$ is a valid mapping.
 First, as each basic function $\func \in \uFunc_{t+1}$ is
 $L_{V_{t+1}}$-Lipschitz continuous on $X_t$, we show that
 $\uvopt_{\uFunc_{t+1}}$ is $L_{V_{t+1}}$-Lipschitz continuous on $X_t$ as
 well. Given $x_1, x_2 \in X_t$, we have
 \begin{align*}
  \lvert \uvopt_{\uFunc_{t+1}}\np{x_1} - \uvopt_{\uFunc_{t+1}}\np{x_2} \rvert & = \lvert \inf_{\func \in \uFunc_{t+1}} \func\np{x_1} - \inf_{\func \in \uFunc_{t+1}} \func\np{x_2} \rvert \\
                                                                              & \leq \sup_{\func \in \uFunc_{t+1}} \lvert \func\np{x_1} - \func\np{x_2} \rvert                            \\
                                                                              & \leq  L_{V_t} \lVert x_1 - x_2 \rVert.
 \end{align*}
 As the Bellman operator $\aB{}{}$ is Lipschitz regular in the sense of Proposition~\ref{lipschitz_regularity}, $\aB{\uvopt_{\uFunc_{t+1}}}{}$ is $L_{V_t}$-Lipschitz continuous.

 Second, by min-additivity of the Bellman operator $\aB{}{}$, we have that
 \[
  \aB{\uvopt_{\uFunc_{t+1}}}{x} = \aB{\inf_{\func \in \uFunc_{t+1}}\func}{x} = \inf_{\func \in \uFunc_{t+1}} \aB{\func}{x}.
 \]
 Recall that by Lemma~\ref{polypreserved}, the Bellman operator $\aB{}{}$ preserves polyhedrality. As $\func \in \uFunc_{t+1}$ is polyhedral, $\aB{\func}{}$ is polyhedral as well and as in the case $t=T$, \emph{mutatis mutandis} we have that $\overline{S}_{t}^{\mathrm{V}}$ is valid.
\end{proof}

\begin{figure}
 \begin{center}
  \begin{subfigure}[b]{0.30\textwidth}
   \includegraphics[width=\linewidth]{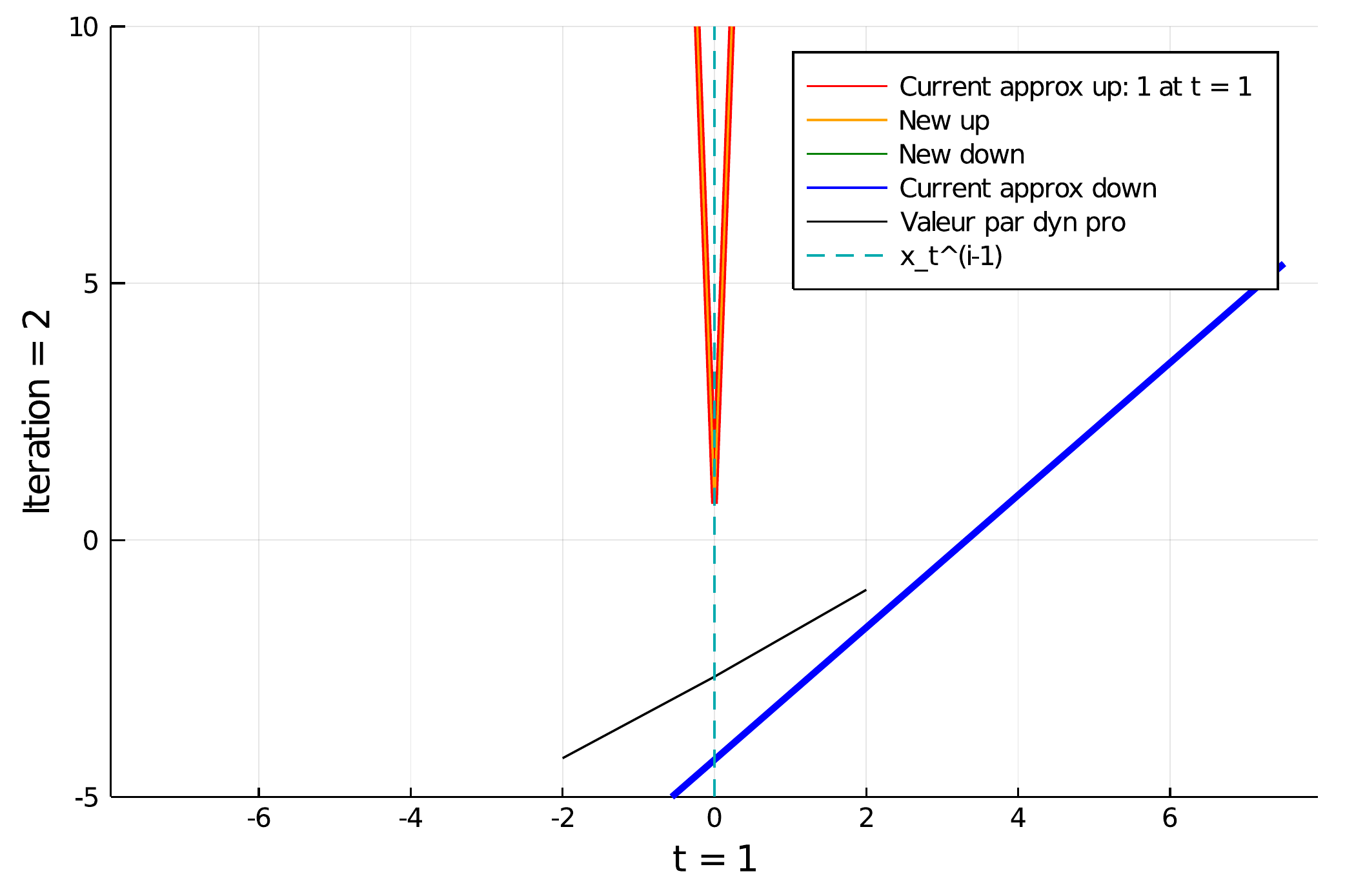}
  \end{subfigure}
  \begin{subfigure}[b]{0.3\textwidth}
   \includegraphics[width=\linewidth]{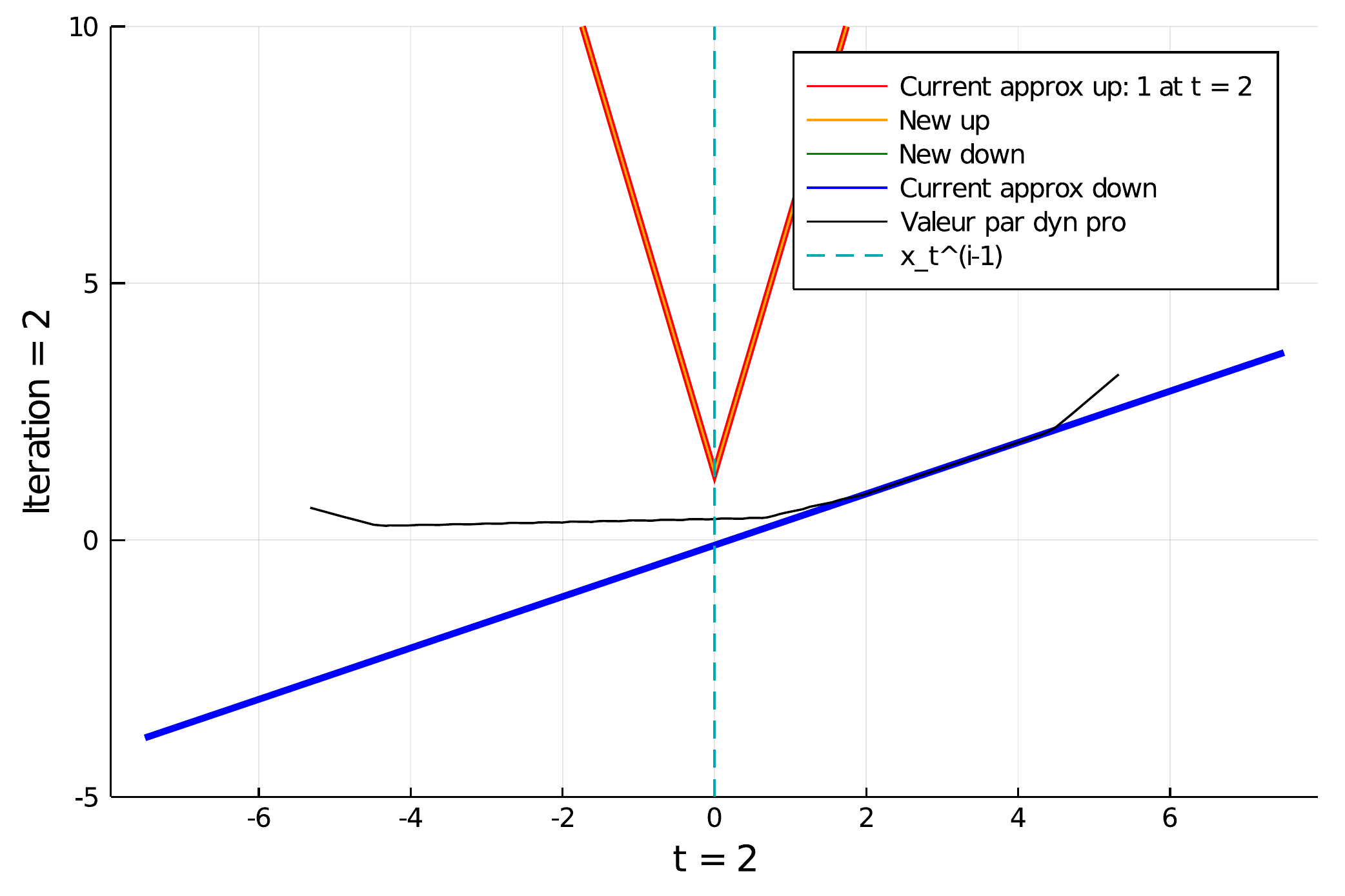}
  \end{subfigure}
  \begin{subfigure}[b]{0.3\textwidth}
   \includegraphics[width=\linewidth]{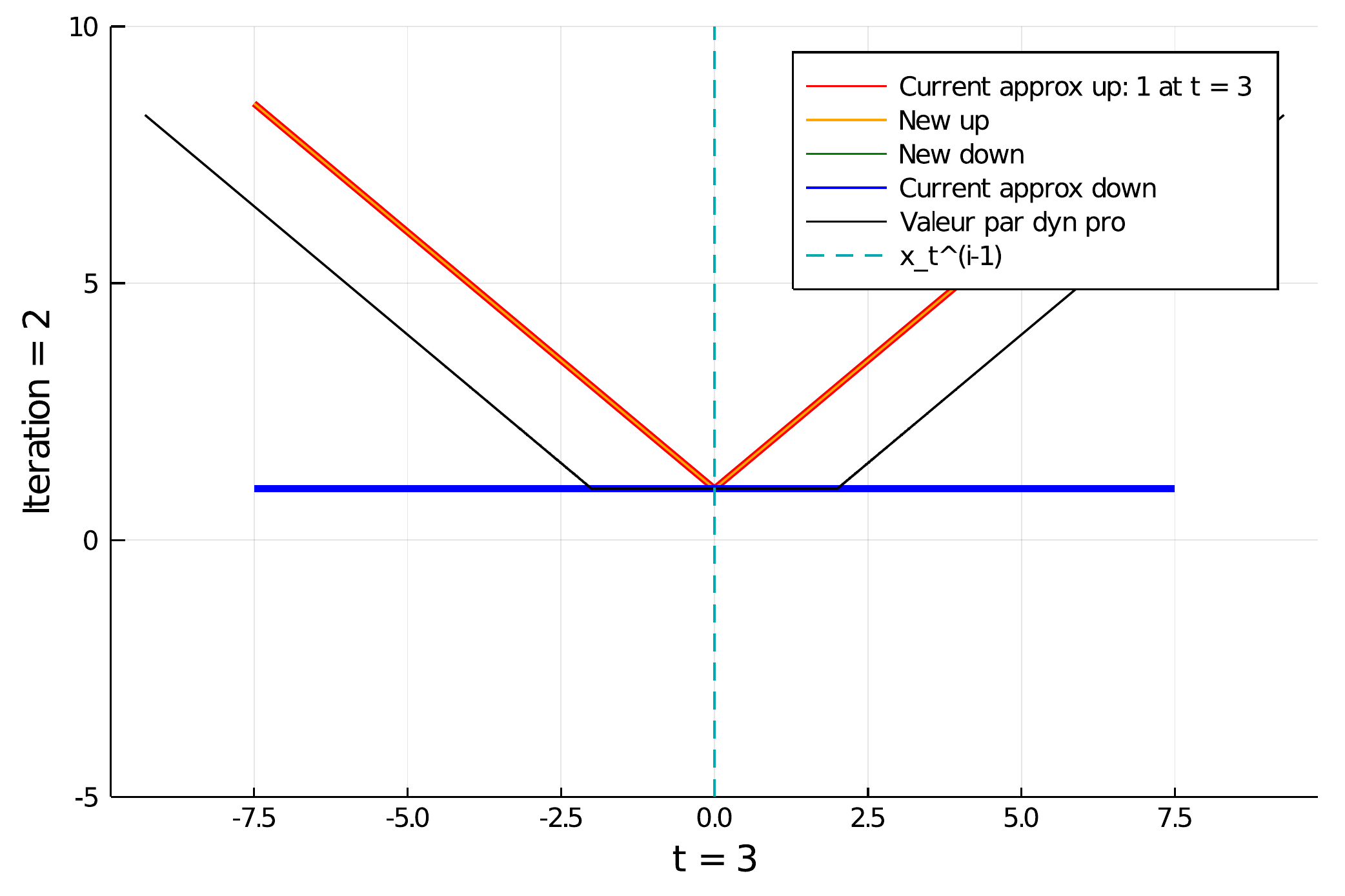}
  \end{subfigure}

  \begin{subfigure}[b]{0.30\textwidth}
   \includegraphics[width=\linewidth]{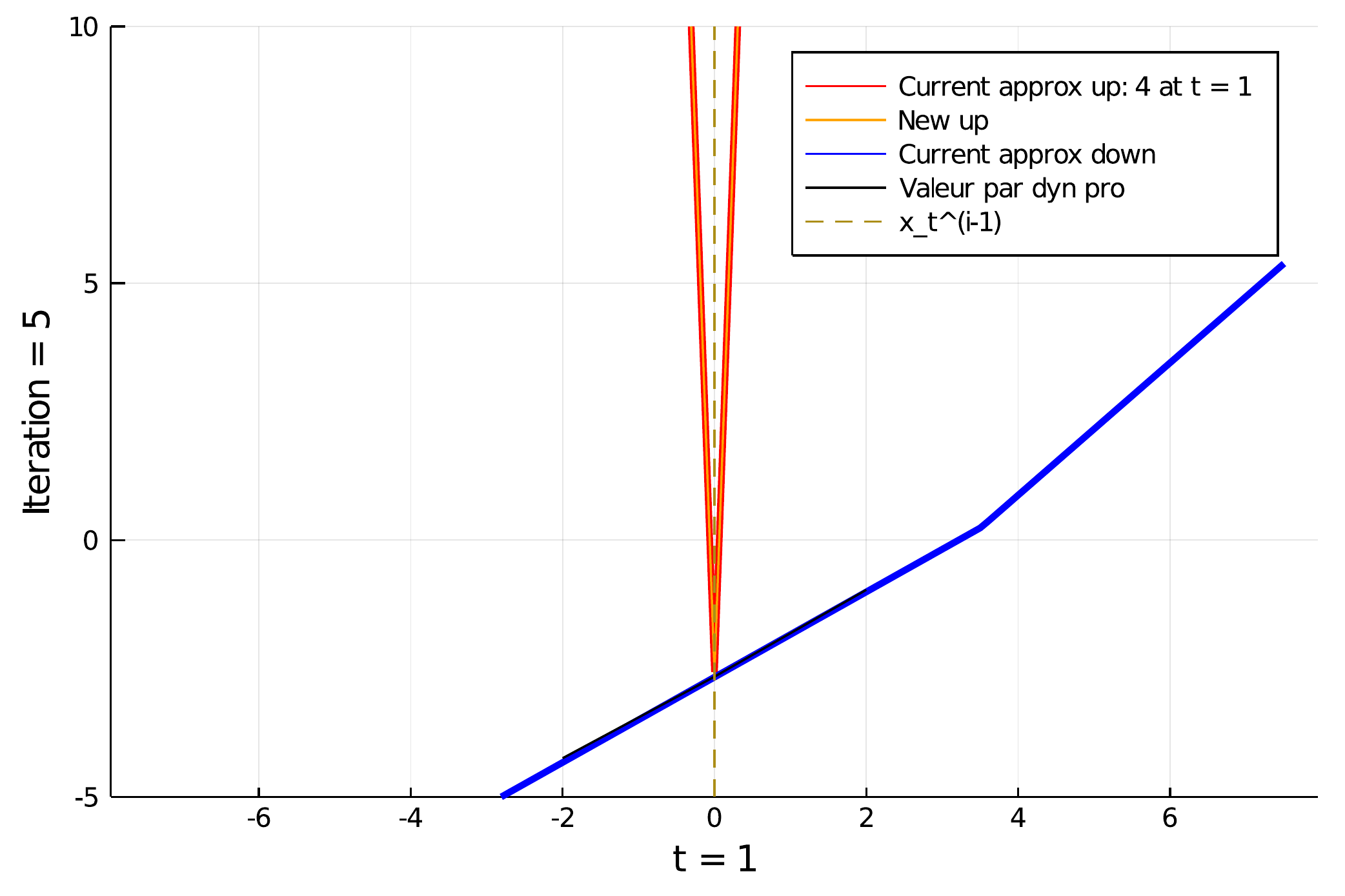}
  \end{subfigure}
  \begin{subfigure}[b]{0.3\textwidth}
   \includegraphics[width=\linewidth]{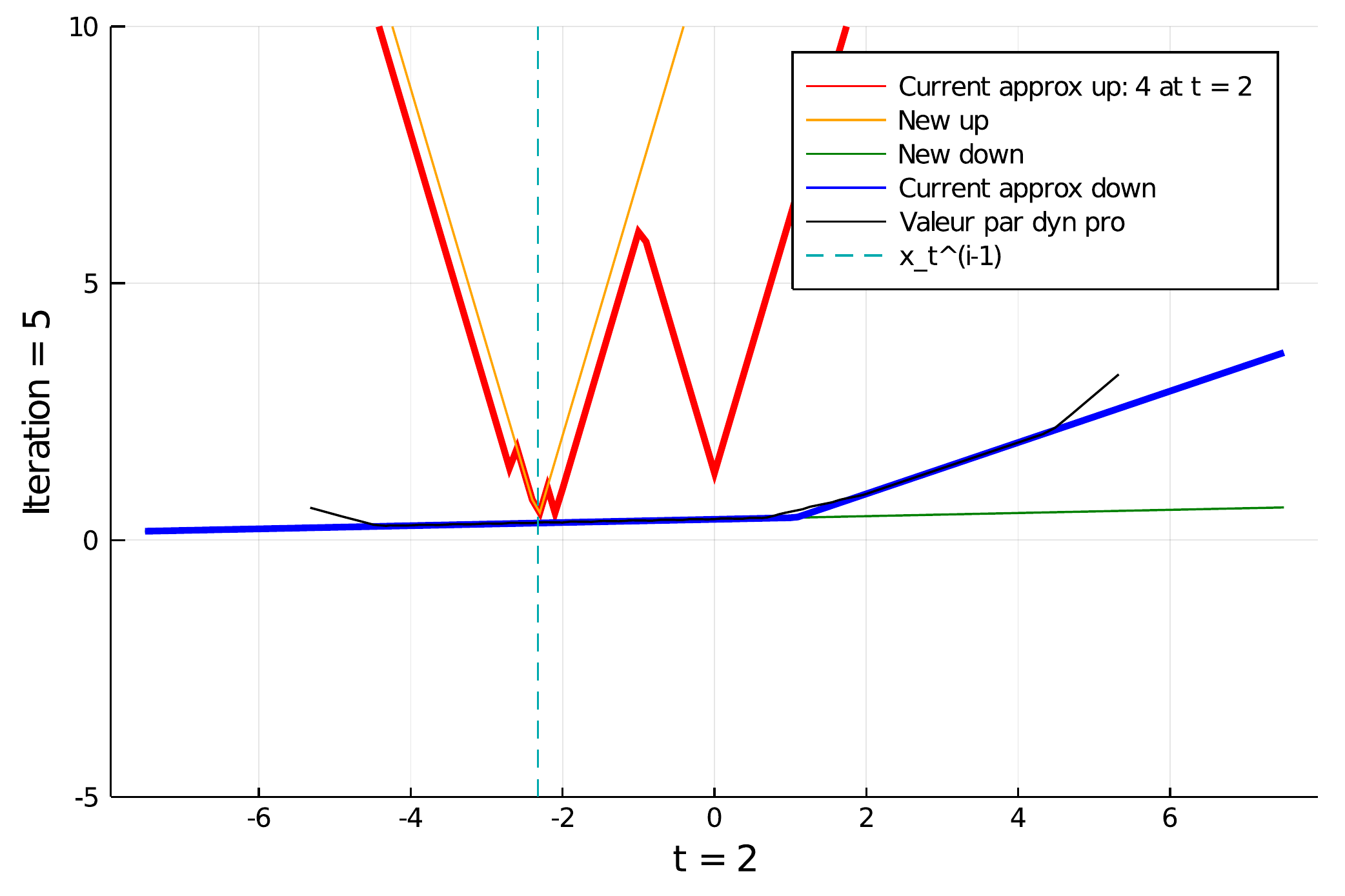}
  \end{subfigure}
  \begin{subfigure}[b]{0.3\textwidth}
   \includegraphics[width=\linewidth]{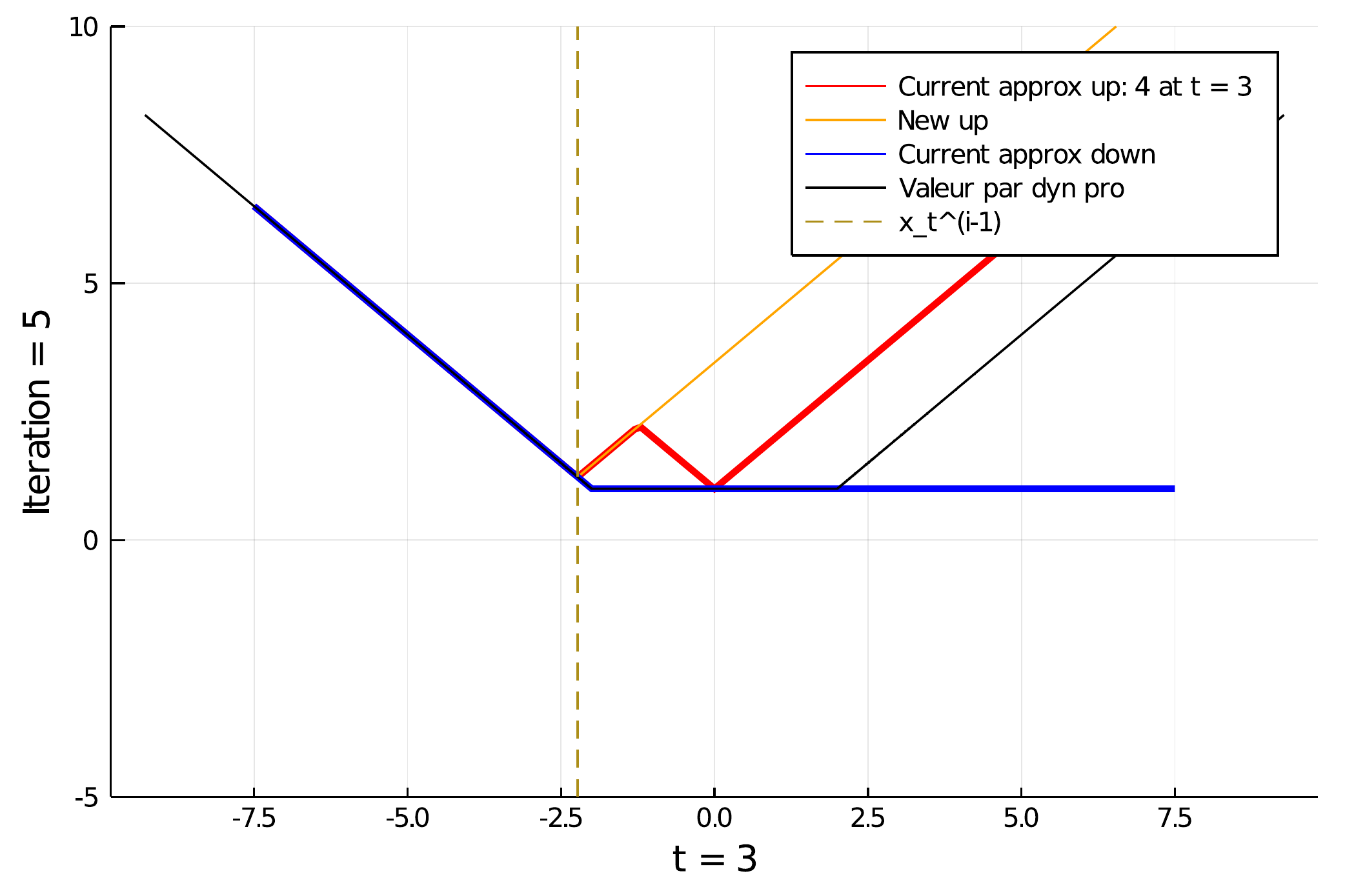}
  \end{subfigure}

  \begin{subfigure}[b]{0.30\textwidth}
   \includegraphics[width=\linewidth]{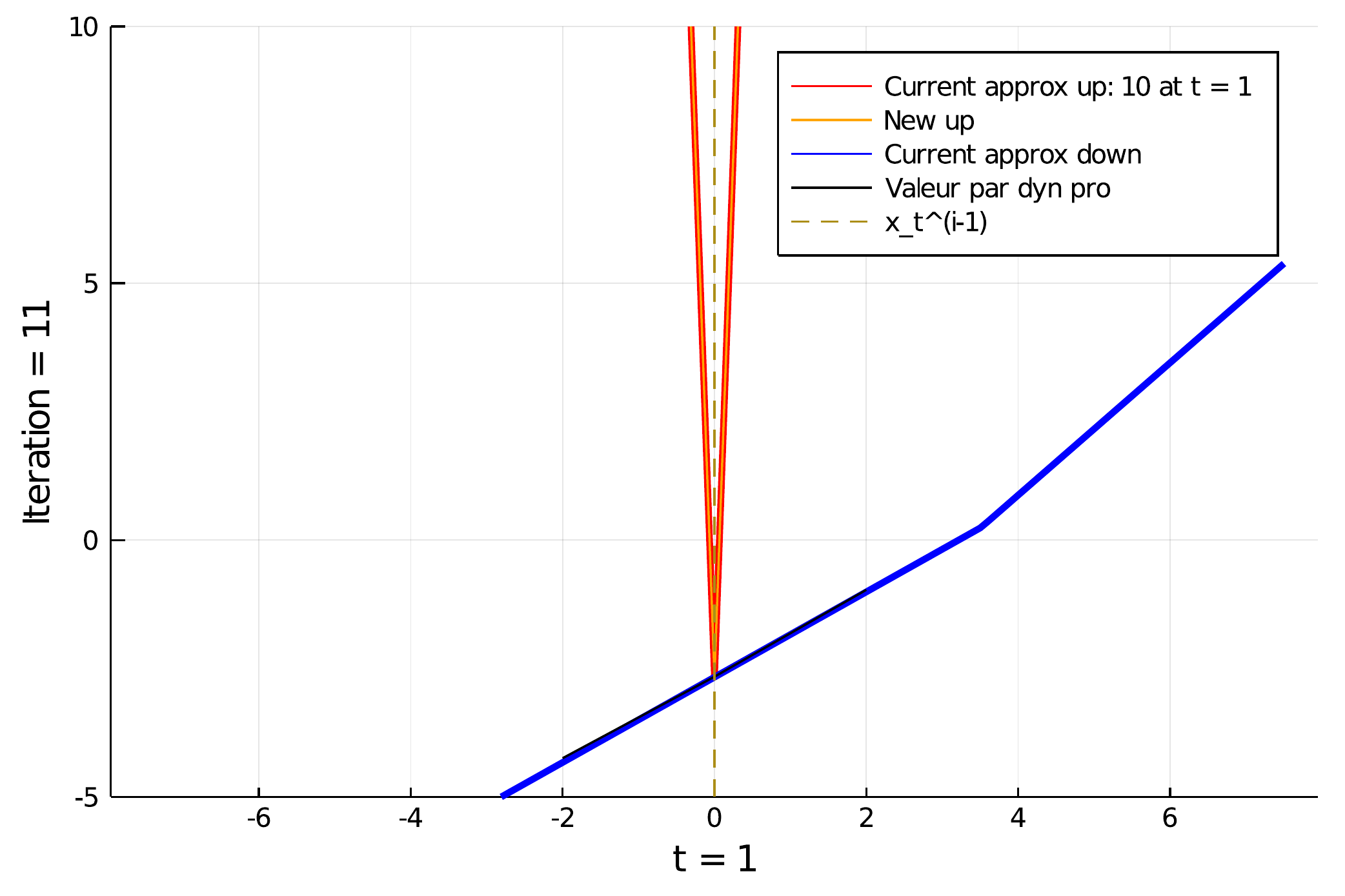}
  \end{subfigure}
  \begin{subfigure}[b]{0.3\textwidth}
   \includegraphics[width=\linewidth]{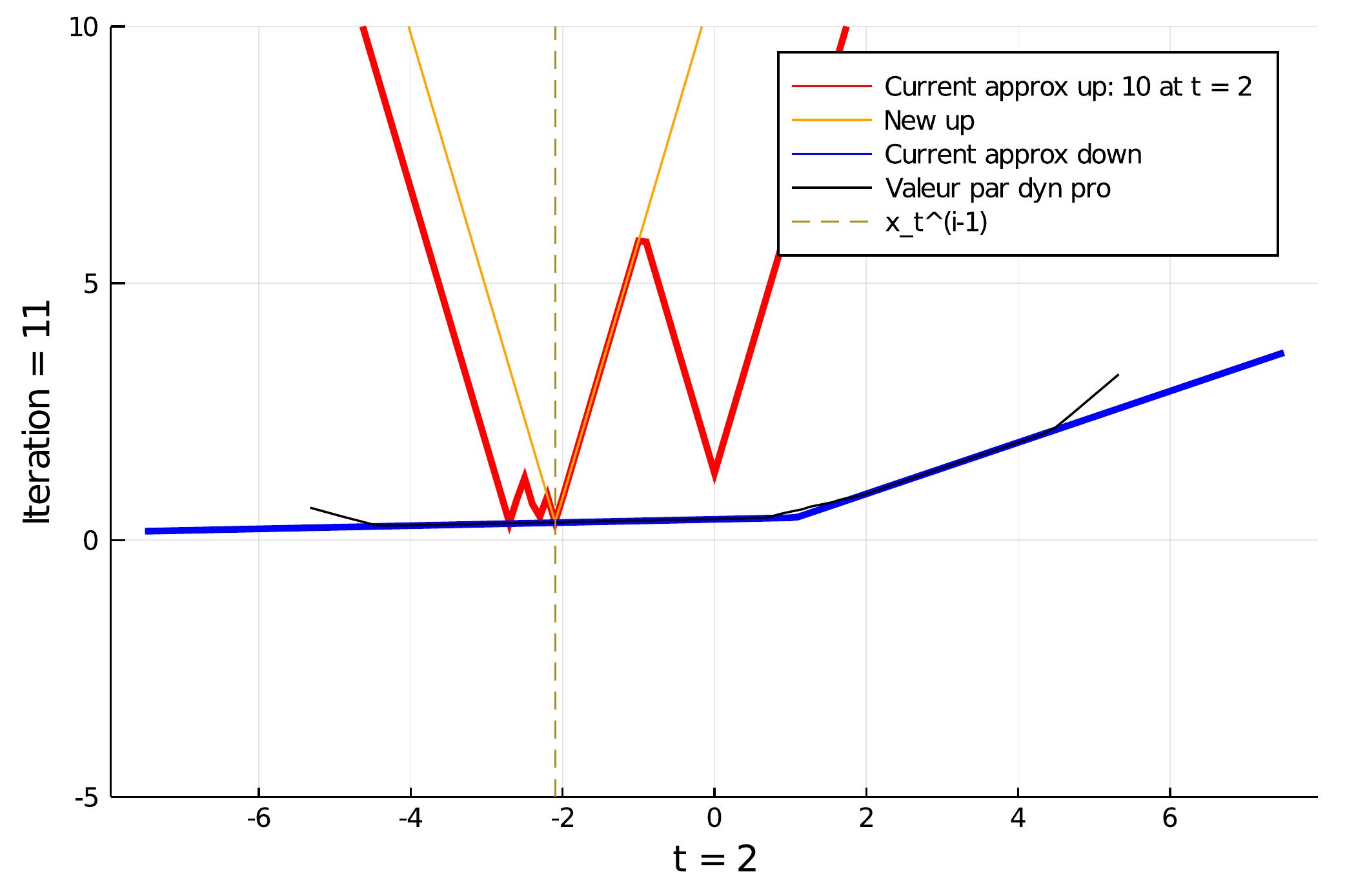}
  \end{subfigure}
  \begin{subfigure}[b]{0.3\textwidth}
   \includegraphics[width=\linewidth]{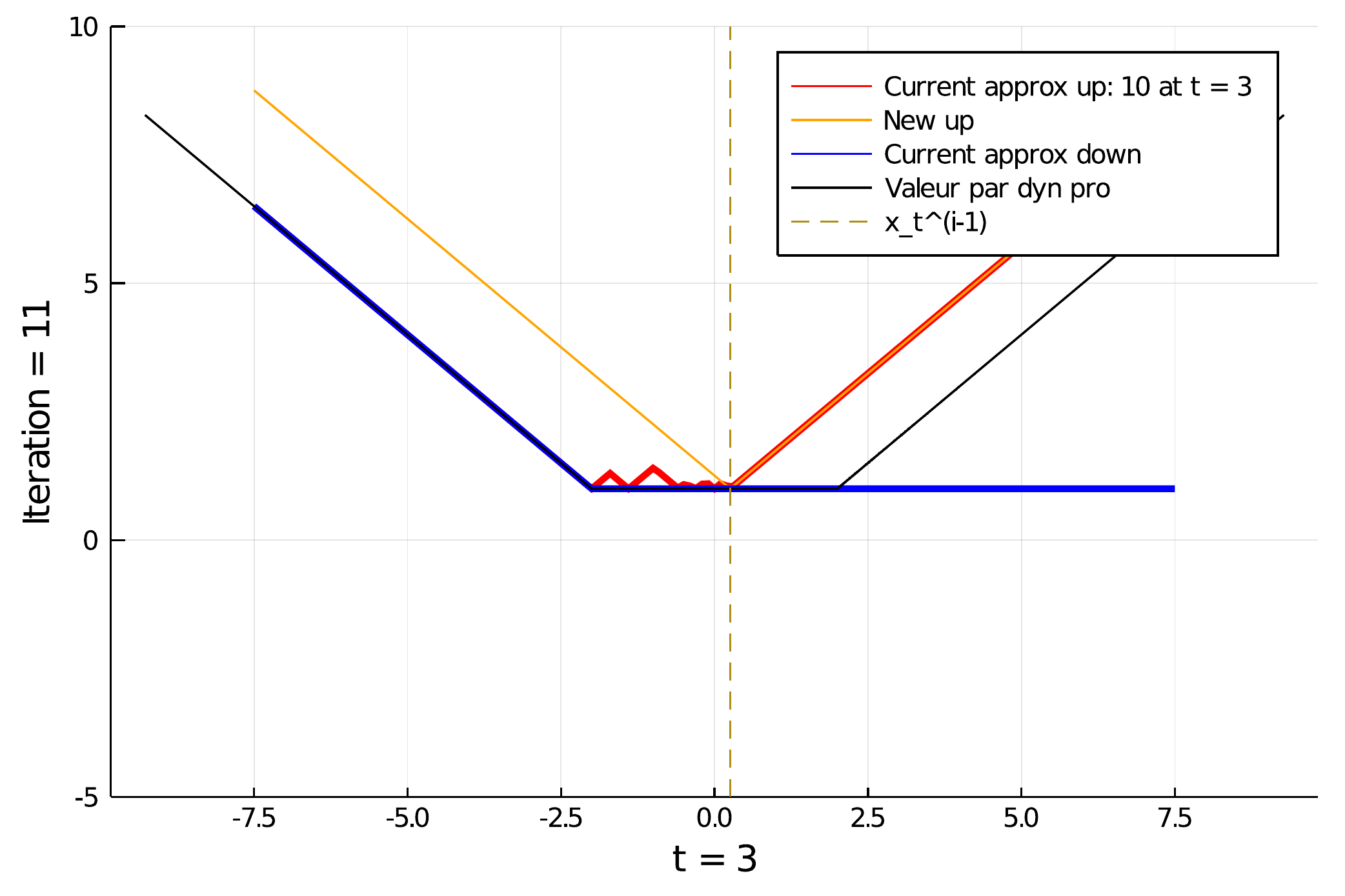}
  \end{subfigure}
 \end{center}
 \caption{V-SDDP approximations of the value functions. As the selection function $S_t^{\mathrm{V}}$ does not average other basic functions to compute a new one (compare with $S_t^{\mathrm{U}}$ or $S_t^{\mathrm{SDDP}}$), we lose the regularizing effect of averaging: the upper basic functions added are very sharp. We still observe that the gap between upper and lower approximations vanishes along the problem-child trajectory (in dashed lines).}
\end{figure}

\section*{Conclusion}

\begin{itemize}
 \item TDP generates simultaneously \emph{monotonic} approximations $\np{\lV_t^k}_k$ and $\np{\uV_t^k}_k$ of $V_t$.
 \item Each approximation is either a \emph{min-plus} or \emph{max-plus linear} combinations of basic functions.
 \item Each basic function should be \emph{tight} and \emph{valid}.
 \item The approximations are refined iteratively along the Problem-child trajectory \emph{without discretizing the state space}.
 \item The \emph{gap} between upper and lower approximation \emph{vanishes along the Problem-child trajectory}.
 \item TDP generalizes a similar approach done in \cite{Ph.de.Fi2013} and proved by \cite{Ba.Do.Za2018} for a variant of SDDP in convex MSPs.
\end{itemize}

\section*{Perpectives}

\begin{itemize}
 \item Consider an additional min-plus/max-plus projection step of suprema/infima of basic functions.
 \item Extensive numerical comparisons with existing methods, namely classical SDDP and the upper approximations obtained by Fenchel duality of \cite{Le.Ca.Ch.Le.Pa2018}.
 \item Extend the scope of TDP to encompass Partially Observed Markov Decision Processes. A first attempt to do so can be found in Appendix~\ref{TDP_POMDP}.
       % Supprimer la reference a l'appendice de la these.
\end{itemize}

% This part can be excluded at compilation
% using the comment package
% \excludecomment{excludepomdp}
% \includecomment{insertpomdo}
% then
% \begin{excludepomdp}
%   ...
% \end{excludepomdp}

\appendix

\section{Tropical Dynamic Programming for POMDP}
\label{TDP_POMDP}
In this section, we present an on-going work to apply TDP on Partially Observed Markov Decision Processes (POMDP).

\subsection{Recalls on POMDP}

Formally, a POMDP is described (in the finite settings) by a finite set of states $\STATESPACE = \na{\state_1, . . . , \state_{|\STATESPACE|}}$,
a finite set of actions $\CONTROL = \na{\control_1, . . . , \control_{|\CONTROL|}}$, a finite set of observations
$\OBSERVATION = \na{\observation_1, . . . , \observation_{|\OBSERVATION|}}$, transition probabilities of the Markov chain
\begin{equation}
 \TransitionState{\control}{\state_i}{\state_j}{t}
 = \nprobc{\State_{t+1}=\state_j}{\State_t=\state_i,\Control_t=\control}
 \eqfinv
\end{equation}
and conditional law of the observations
\begin{equation}
 \ObservationLaw{\observation}{\state}{\control}{t+1}
 = \nprobc{\Observation_{t+1}=\observation}{\State_{t+1}=\state,\Control_t=\control}
 \eqfinv
\end{equation}
a real-valued cost function $L_t(\state, \control)$ for any $t \in \ic{0,T-1}$, a final cost $K(x)$
and an initial probability law in the simplex of $\RR^{|\STATESPACE|}$ called the initial belief $\belief_0$.
We assume here that the state space the control space and the observation space dimensions do not vary with time
but for the sake of clarity we will use the notation $\STATESPACE_t$ to designate the state space at time $t$ even if it is equal to $\STATESPACE$ and the same for control and observation states.

Under Markov assumptions, we can use at time $t$ a probability distribution
$\belief_t$, whose name is a reminder of belief, over current states as a
sufficient statistic for the history of actions and observations up to time
$t$. The space of beliefs is the simplex of $\RR^{|\STATESPACE|}$, denoted
$\Delta_{|\STATESPACE|}$. The belief dynamics, at time $t$, driven by action
$\control_t$ and observation $\observation_{t+1}$ is given by by the equation
\begin{align}
 \belief_{t+1}
  & = \tau_t \np{\belief_{t}, \control_t, \observation_{t+1}}
 \intertext{with $\belief_{t+1} \in \Delta_{|\STATESPACE|}$ given by} % \\
 \belief_{t+1}\np{\state_{t+1}}
  & =
 \beta_{t+1} \ObservationLaw{\observation_{t+1}}{\state_{t+1}}{\control_t}{t+1}
 \Bp{\sum_{x_t \in \STATESPACE_{t}} \belief\np{\state_t} \TransitionState{\control_t}{\state_t}{\state_{t+1}}{t}}
 \quad \forall \state_{t+1} \in  \STATESPACE_{t+1}
 \eqfinv
\end{align}
where $\beta_{t+1}$ is a normalization constant to ensure that
$\belief_{t+1}\in \Delta_{|\STATESPACE|}$, that is
\[
 \beta_{t+1}^{-1} = \sum_{\state_{t+1} \in \STATESPACE_{t+1}} \ObservationLaw{\observation_{t+1}}{\state_{t+1}}{\control_t}{t+1}
 \Bp{\sum_{x_t \in \STATESPACE_t} \TransitionState{\control_t}{\state_t}{\state_{t+1}}{t} \belief\np{\state_t}} \eqfinp
\]
To simplify the notation we introduce the (sub-stochastic) matrix defined as follows
\[
 {M^{\control_t, \observation_{t+1}}_t}\np{\state_t,\state_{t+1}} =
 \ObservationLaw{\observation_{t+1}}{\state_{t+1}}{\control_t}{t+1}
 \TransitionState{\control_t}{\state_t}{\state_{t+1}}{t}
 \quad \forall (\state_t,\state_{t+1}) \in \STATESPACE_t{\times}\STATESPACE_{t+1}
 \eqfinv
\]
where we have $\sum_{\observation_{t+1}} \sum_{\state_{t+1}} {M^{\control_t, \observation_{t+1}}_t}\np{\state_t,\state_{t+1}}=1$.
Using matrix notations, where beliefs are represented by row vector and $\mathbf{1}$ is a column vector full of ones,  we can rewrite the beliefs
dynamics as
\[
 \tau_t \np{\belief_{t}, \control_t, \observation_{t+1}} =
 \frac{ \belief_t {M^{\control_t, \observation_{t+1}}_t}}{  \belief_t {M^{\control_t, \observation_{t+1}}_t} \mathbf{1}}
 \in \Delta_{|\STATESPACE|}
 \eqfinp
\]

In general the object of the optimization problem is to generate a policy that minimizes expected finite horizon cost
for the controlled Markov chain $\na{\State_t^u}_{t\in \NN}$ with transition matrix $\TransitionStateName^{\control}$. That is
consider the minimization problem
\begin{equation}
 J(\belief_0) = \min_{\Control_1,\ldots,\Control_{T-1}} \Bespc{\sum_{t=0}^{T-1} L_t\np{\State_t,\Control_t} + K\np{\State_T}}
 {\belief_0}
 \eqfinp
 \label{pb:pomdp}
\end{equation}

It is classical to derive a Bellman equation for the beliefs given by the bellman operators for $t\in \ic{0,T-1}$
\begin{equation}
 \label{eq:bellman-pomdp}
 \B_t\np{V} = \inf_{\control \in \CONTROL} \B^{\control}_t \np {V}
 \eqfinv
\end{equation}
where for each $\control \in \CONTROL$ and $t \in \ic{0,T-1}$, the Bellman operator $\B^{\control}_{t}$ is defined by
% \begin{align*}
%  \B^{\control}_{t}\np{V}\np{\belief}
%   & = \sum_{\state_t\in \STATESPACE_t} L_t\np{\control,\state_t} \belief\np{\state_t}
%  \\
%   & \hspace{0.5cm}
%  +
%  \sum_{\observation_{t+1} \in \Observation_{t+1}}
%  {\sum_{\state_t \in \STATESPACE_{t}, \state_{t+1}\in \STATESPACE_{t+1} }
%   \ObservationLaw{\observation_{t+1}}{\state_{t+1}}{\control_t}{t+1}
%   \TransitionState{\control_t}{\state_t}{\state_{t+1}}{t}
%   \belief\np{\state_t}}
%  V\bp{\tau_t\np{\belief_{t}, \control_t, \observation_{t+1}}}
%       \eqfinp
% \end{align*}
\begin{align}
 \label{eq:bellman-u-pomdp}
 \B^{\control}_{t}\np{V}\np{\belief}
  & =
 \belief L_t^\control
 +  \sum_{\observation \in \OBSERVATION_{t+1}}
 \np{\belief M^{\control, \observation}_t \mathbf{1}}
 V
 \Bp{
  \frac{ \belief M^{\control, \observation}_t}
  {\belief M^{\control, \observation}_t \mathbf{1}}}
 \eqfinp
\end{align}
where $L_t^u$ is the column vector $\bp{L_t^u(\state_t)}_{\state_t \in \STATESPACE_t}$. Note that
the mapping $\observation_{t+1} \in \OBSERVATION_{t+1} \mapsto \np{\belief_t M^{\control_t, \observation_{t+1}}_t \mathbf{1}}$ is a
probability distribution on $\OBSERVATION_{t+1}$ ($\sum_{\observation \in \OBSERVATION_{t+1}} \belief_t M^{\control_t, \observation_{t+1}}_t \mathbf{1}=1$).

The Bellman operator can be also written as
\begin{align}
 \B^{\control}_{t}\np{V}\np{\belief}
  & =
 \belief L_t^u
 +  \sum_{\belief' \in \Delta_{|\STATESPACE|}} \overline{P}^u_t(\belief, \belief') V \np{\belief'}
 \eqfinv
\end{align}
where, $\overline{P}^u$ is a controlled Markov chain transition matrix in the belief space. Indeed
\begin{equation}
 \overline{P}^u_t(\belief, \belief')
 =
 \begin{cases}
  \np{\belief M^{\control, \observation}_t \mathbf{1}} & \text{when }  \belief' =
  \frac{ \belief_t {M^{\control, \observation}_t}}{  \belief_t {M^{\control, \observation}_t} \mathbf{1}} \text{ with } \observation \in \OBSERVATION_{t+1} \eqfinv
  \\
  0                                                    & \text{ if not}\eqfinv
 \end{cases}
\end{equation}
which is a classical Bellman equation of a controlled Markov chain but with a state space in the belief space.

We conclude this section by the following lemma
\begin{proposition} The value functions $\na{V_t}_{t\in \ic{0,T}}$ solutions of the Bellman Equation
 \begin{equation}
  \forall b \in\RR_{+}^{|\STATESPACE|}\quad {V}_T(b) = b K  \quad\text{and}\quad
  \forall t \in \ic{0,T{-}1}\quad
  {V}_t(b) = \inf_{\control \in \CONTROL} {\B}^{\control}_{t}\np{{V}_{t+1}}\np{b}
  \eqfinv
  \label{eq:bellman-eq}
 \end{equation}
 where the operator ${\B}^{\control}_{t}$ is given by Equation~\ref{eq:bellman-u-pomdp}
 are such that $V_0(b_0)$ is the optimal value of the minimization problem given by Equation~\ref{pb:pomdp}.
\end{proposition}

\subsection{The Bellman operator defined in Equation~\eqref{eq:bellman-pomdp} propagate Lipschitz mappings}

\begin{proposition} For $t\in \ic{0,T{-}1}$, assume that the mappings $L_t(\control,\cdot)$
 satisfy $\norm{L_t(\control,\cdot)}_{\infty} \le {\cal L}$\footnote{Since the state space if finite we
  identify mappings $\phi:\STATESPACE\to \RR$ with vectors in $\RR^{|\STATESPACE|}$}
 for all $\control \in \CONTROL$ and assume that a mapping $K$ satisfy $\sup_{\state \in \STATESPACE} |K(x)|= {\cal K}< + \infty$.
 Then the solution of the Bellman Equation~\eqref{eq:bellman-eq} are Lipschitz mappings.
\end{proposition}
\begin{proof}\quad

 \noindent $\bullet$ We consider the operator $\widetilde{\B}^{\control}_{t}$
 defined for mappings $\widetilde{V}: \RR_{+}^{|\STATESPACE|} \to \RR$ by
 \begin{align}
  \widetilde{\B}^{\control}_{t}\np{\widetilde{V}}\np{c}
   & =
  c L_t^\control
  +  \sum_{\observation \in \OBSERVATION_{t+1}}
  \widetilde{V}
  \bp{c M^{\control, \observation}_t} \quad \forall c \in \RR^{|\STATESPACE|}
  \eqfinv
 \end{align}
 where $L^u_t$ stands for the column vector $\np{L_t(x,u)}_{x \in \STATESPACE_t}$
 and we recall that beliefs are row vectors. We consider $\na{\widetilde{V_t}}_{t\in \ic{0,T}}$ solution of the
 Bellman Equation
 \begin{equation}
  \forall c \in\RR_{+}^{|\STATESPACE|}\quad \widetilde{V}_T(c) = cK  \quad\text{and}\quad
  \forall t \in \ic{0,T{-}1}\quad
  \widetilde{V}_t(c) = \inf_{\control \in \CONTROL} \widetilde{\B}^{\control}_{t}\np{\widetilde{V}_{t+1}}\np{c}
  \eqfinp
  \label{eq:modified-bellman-pomdp}
 \end{equation}
 First, we straightforwardly obtain by backward induction that the value functions
 $\np{\widetilde{V}_t}_{t \in \ic{0,T}}$ are homogeneous of degree $1$. Second
 we prove that the operator $\widetilde{\B}^{\control}_{t}$ preserves Lispchitz
 regularity.  We proceed as follows. Consider $c$ and $c'$ in
 $\RR_{+}^{|\STATESPACE|}$ and suppose that
 $| \widetilde{V}(c)-\widetilde{V}(c')| \le {\cal V} \norm{c'-c}_1$. Then we
 have that
 \begin{align*}
  \widetilde{\B}^{\control}_{t}\np{\widetilde{V}}\np{c'} -
  \widetilde{\B}^{\control}_{t}\np{\widetilde{V}}\np{c}
   & = (c' - c) L_t^\control
  +  \sum_{\observation \in \OBSERVATION_{t+1}} \widetilde{V}\bp{c' M^{\control, \observation}_t} -
  V\bp{ c M^{\control, \observation}_t}
  \\
   & \le {\cal L}\norm{c'-c}_1
  + \sum_{\observation \in \OBSERVATION_{t+1}} {\cal V}
  \norm{c' M^{\control, \observation}_t - cM^{\control, \observation}_t}_1
  \\
   & \le {\cal L}\norm{c'-c}_1
  + {\cal V}
  \sum_{\substack{\observation \in \OBSERVATION_{t+1} \\ \state' \in \STATESPACE}}
  \Big\lvert
  \sum_{\state \in \STATESPACE}
  \bp{c'(\state)-c(\state)} M^{\control, \observation}_t(\state,\state')
  \Big\rvert
  \\
   & \le {\cal L}\norm{c'-c}_1
  + {\cal V}
  \sum_{\state \in \STATESPACE}
  |c'(\state)-c(\state)|
  \sum_{\substack{\observation \in \OBSERVATION_{t+1} \\ \state' \in \STATESPACE}}
  M^{\control, \observation}_t(\state,\state')
  \\
   & \le {\cal L}\norm{c'-c}_1
  + {\cal V}
  \sum_{\state \in \STATESPACE}
  |c'(\state)-c(\state)|
  \\
   & \le
  \Bp{{\cal L} +{\cal V}} \norm{c'-c}_1
  \eqfinp
 \end{align*}

 As a pointwise minimum of Lipschitz mappings having the same Lipschitz constant is Lipschitz, we obtain the same Lispchitz
 constant for the operators $\inf_{\control \in \CONTROL}\widetilde{\B}^{\control}_{t}$. Then, using the fact that
 $\overline{V}_T = K$ we obtain by backward induction that the Bellman value function $\widetilde{V}_{t}$ is
 $({\cal L}(T-t) + {\cal K})$-Lipschitz
 for $t\in \ic{0,T}$  where ${\cal K}= \norm{K(\cdot)}_{\infty}$.

 \noindent $\bullet$ We prove now an intermediate result to link the solutions of the Bellman Equation~\eqref{eq:modified-bellman-pomdp}
 to the Bellman Equation~\eqref{eq:bellman-eq}. Suppose that $\widetilde{V}$ is
 $1$-homogeneous and such that $\widetilde{V}(b) = V(b)$ for all
 $b \in \Delta_{|\STATESPACE|}$. Then, We prove that
 $\widetilde{\B}^{\control}_{t}\np{\widetilde{V}}(b)=
  {\B}^{\control}_{t}\np{{V}}(b)$ for all $b \in \Delta_{|\STATESPACE|}$. For
 $b\in \Delta_{|\STATESPACE|}$, we successively have that
 \begin{align}
  \widetilde{\B}^{\control}_{t}\np{\widetilde{V}}\np{b}
   & =
  b L_t^\control
  +  \sum_{\observation \in \OBSERVATION_{t+1}}
  \widetilde{V}
  \bp{b M^{\control, \observation}_t}
  \\
   & =      b L_t^\control
  +  \sum_{\observation \in \OBSERVATION_{t+1}}
  \np{b M^{\control, \observation}_t\mathbf{1}}
  \widetilde{V}
  \bp{ \frac{b M^{\control, \observation}_t}{b M^{\control, \observation}_t\mathbf{1}}}
  \tag{$\widetilde{V}$ is $1$-homogeneous}
  \\
   & =
  b L_t^\control
  +  \sum_{\observation \in \OBSERVATION_{t+1}}
  \np{b M^{\control, \observation}_t\mathbf{1}}
  {V}
  \bp{ \frac{b M^{\control, \observation}_t}{b M^{\control, \observation}_t\mathbf{1}}}
  \tag{$\widetilde{V}=V$ on $\Delta_{|\STATESPACE|}$}
  \\
   & = {\B}^{\control}_{t}\np{{V}}\np{b}
  \eqfinp
 \end{align}

 \noindent $\bullet$ Now we turn to solutions of Bellman Equation~\eqref{eq:bellman-eq}.
 Since $\widetilde{V}_T(c) = c K$ for all $c\in \RR_{+}^{|\STATESPACE|}$ and
 $V_T(b) = bK$ for all $b \in \Delta_{|\STATESPACE|}$, the two mappings $V_T$ and $\widetilde{V}_T$ coincide
 on the simplex of dimension $|\STATESPACE|$. Then gathering the previous steps we obtain that
 $V_t$ and $\widetilde{V}_t$ coincide
 also on the simplex of dimension $|\STATESPACE|$ for all $t \in \ic{0,T}$.
 Finally, for all $t \in \ic{0,T}$ $\widetilde{V}_t$ being $({\cal L}(T-t) + {\cal K})$-Lipschitz
 we obtain the same result for $V_t$.
\end{proof}

\subsection{Value of  $\B_{t}\np{V_{t+1}}$ when $V_{t+1} = \min_{\alpha \in \Gamma_{t+1}} \proscal{\alpha}{\belief}$}

Assume that $V_{t+1}: \belief \mapsto \min_{\alpha \in \Gamma_{t+1}}
 \proscal{\alpha}{\belief}$ where $\Gamma_{t+1} \subset
 \RR^{|\STATESPACE|}$. Then we obtain that

\begin{align}
 \B_{t}\np{V_{t+1}}\np{\belief}
  & =
 \min_{\control \in \CONTROL_t}
 \Bp{{\belief L_t^u}
  +  \sum_{\observation \in \OBSERVATION_{t+1}}
  \np{\belief_t M^{\control, \observation}_t \mathbf{1}}
  V_{t+1}\Bp{
   \frac{ \belief {M^{\control, \observation}_t}}{\belief {M^{\control, \observation}_t} \mathbf{1}}
  }}
 \\
  & =
 \min_{\control \in \CONTROL_t}
 \Bp{{\belief L_t^u}
  +  \sum_{\observation \in \OBSERVATION_{t+1}}
  \np{\belief M^{\control, \observation}_t \mathbf{1}}
  \min_{\alpha \in \Gamma_{t+1}} \Bp{
   \frac{ \belief {M^{\control, \observation}_t} \alpha }{  \belief {M^{\control, \observation}_t} \mathbf{1}
   }}}
 \\
  & =
 \min_{\control \in \CONTROL_t}
 \Bp{{\belief L_t^u}
  +  \sum_{\observation \in \OBSERVATION_{t+1}}
  { \belief {M^{\control, \observation}_t} \alpha^{\sharp}({\control, \observation}) }
  \tag{with $\alpha^{\sharp}({\control, \observation})
    = \argmin_{\alpha \in \Gamma_{t+1}}
    \frac{ \belief {M^{\control, \observation}_t} \alpha }
    {  \belief {M^{\control, \observation}_t} \mathbf{1}}$}
 }
 \\
  & =
 \min_{\control \in \CONTROL_t}\belief
 \Bp{ L_t^u
  +  \sum_{\observation \in \OBSERVATION_{t+1}}
  { {M^{\control, \observation}_t} \alpha^{\sharp}({\control, \observation}) }}
 \\
  & =
 \min_{\alpha \in  \Gamma_t} \proscal{\alpha}{\belief}
 \eqfinv
\end{align}
with
$\Gamma_t = \bset{{ L_t^u + \sum_{\observation \in \OBSERVATION_{t+1}} {
     {M^{\control, \observation}_t} \alpha^{\sharp}({\control, \observation})
    }}} {\control \in \CONTROL_t \,\text{and}\, \alpha^{\sharp}({\control,
   \observation}) = \argmin_{\alpha \in \Gamma_{t+1}} \frac{ \belief
   {M^{\control, \observation}_t} \alpha } { \belief {M^{\control,
      \observation}_t} \mathbf{1}}}$.  We therefore obtain that the Bellman
value function at time $t$ has the same form as the Bellman value function at
time $t+1$.

We are in a context where the Bellman function that is to to be computed is
polyhedral concave with a huge polyhedron. It is thus tempting to use our
algorithm with polyhedral concave upper approximations and sup of quadratic or
Lipschitz mappings as lower approximations.

The \emph{Problem-child trajectory} technique is used in POMDP algorithms as an heuristic
but without a convergence proof as far as we have investigated.

\subsection{A lower bound of $\B_{t}\np{V_{t+1}}$}
We consider a special case where ${V}_{t+1}:\STATESPACE \to \RR$ is
given by ${V}_{t+1}(\belief)= \proscal{\belief}{\widehat{V}_{t+1}}$ and we compute
$\B_{t}\np{V_{t+1}}$ as follows
\begin{align*}
 \B_{t}\np{V_{t+1}}\np{\belief}
  & =
 \min_{\control \in \CONTROL_t}
 \Bp{
  {\belief} {L_t^u}
  +  \sum_{\observation \in \OBSERVATION_{t+1}}
  \belief {M^{\control, \observation}_t} \widehat{V}_{t+1}}
 \\
  & =
 \min_{\control \in \CONTROL_t}
 \Bp{ {\belief}L_t^u
  + \sum_{\observation \in \OBSERVATION_{t+1}, \state \in \STATESPACE_t, \state' \in \STATESPACE_{t+1}}
  \ObservationLaw{\observation}{\state'}{\control}{t+1}
  \TransitionState{\control}{\state}{\state'}{t}
  \belief\np{\state}
  \widehat{V}_{t+1}\np{\state'}}
 \\
  & =
 \min_{\control \in \CONTROL_t}
 \Bp{{\belief}{L_t^u}
  + \sum_{ \state \in \STATESPACE_t, \state' \in \STATESPACE_{t+1}}
  % \ObservationLaw{\observation}{\state'}{\control}{t+1}
  \TransitionState{\control}{\state}{\state'}{t}
  \belief\np{\state}
  \widehat{V}_{t+1}\np{\state'}}
 \tag{$\sum_{\observation} \ObservationLaw{\observation}{\state'}{\control}{t+1} = 1$}
 \\
  & \ge
 \sum_{\state \in \STATESPACE_t} \belief\np{\state}
 \min_{\control \in \CONTROL_t} \Bp{L_t\np{\control,\state}
  +\sum_{\state' \in \STATESPACE_{t+1}}
  \TransitionState{\control}{\state}{\state'}{t}
  \widehat{V}_{t+1}\np{\state'}}
 \\
  & =
 \sum_{\state \in \STATESPACE_t} \belief(\state)\widehat{V}_t(\state)
 = \belief \widehat{V}_t
 \eqfinv
\end{align*}
with
\begin{equation}
 \widehat{V}_t(\state)
 = \min_{\control \in \CONTROL_t} \Bp{L_t\np{\control,\state}
  +\sum_{\state' \in \STATESPACE_{t+1}}
  \TransitionState{\control}{\state}{\state'}{t}
  \widehat{V}_{t+1}\np{\state'}}
 \eqfinp
\end{equation}
Using the fact that at time $T$ we have that $V_T =  \proscal{\belief}{\widehat{V}_T}$ with ${\widehat{V}_T}=K$ we obtain that
for all $t\in \ic{0,T}$ $V_t \ge \proscal{\belief}{\widehat{V}_t}$ where ${\widehat{V}_t}$ is the Value function of the fully observed Bellman equation associated to the POMDP.

\bibliographystyle{plain}

\end{document}